\documentclass[11pt]{amsart}
\usepackage{amssymb}
\usepackage{tikz,tikz-cd}
\usetikzlibrary{intersections}
\usepackage{ytableau}
\usepackage[T1]{fontenc}\usepackage{palatino} 
\usepackage[mathscr]{euscript}
\usepackage{stmaryrd} 
\usepackage[pagebackref]{hyperref} 
\hypersetup{colorlinks,linkcolor=blue,urlcolor=blue,citecolor=blue,%
  linktocpage}  
\usepackage[abbrev,nobysame]{amsrefs} 

\newcommand{\Part}[1]{
 \foreach \x [count=\s from 1] in {#1}{
 	{\ifnum\s=1
		\draw (0,\s-1)--(\x,\s-1); 
		\fi}
   \draw (0,\s) to (\x,\s);
   \foreach \y in {0, ..., \x} {\draw (\y,\s)--(\y,\s-1);}
 }}

\def\UNIT{.18} \newcommand{\PART}[1]{
\begin{tikzpicture}[xscale=\UNIT, yscale=-\UNIT] 
	\Part{#1}
\end{tikzpicture}
}
\newcommand{\epic}{
\begin{tikzpicture}[scale = 0.35,thick, baseline={(0,-1ex/2)}]
\tikzstyle{vertex} = [shape = circle, minimum size = 4pt, inner sep = 1pt] 
\node[vertex] (G--2) at (1.5, -1) [shape = circle, draw,fill=black] {}; 
\node[vertex] (G--1) at (0.0, -1) [shape = circle, draw,fill=black] {}; 
\node[vertex] (G-1) at (0.0, 1) [shape = circle, draw,fill=black] {}; 
\node[vertex] (G-2) at (1.5, 1) [shape = circle, draw,fill=black] {}; 
\draw (G--2) .. controls +(-0.5, 0.5) and +(0.5, 0.5) .. (G--1); 
\draw (G-1) .. controls +(0.5, -0.5) and +(-0.5, -0.5) .. (G-2); 
\end{tikzpicture} 
}

\newcommand{\onepic}{
\begin{tikzpicture}[scale = 0.35,thick, baseline={(0,-1ex/2)}] 
\tikzstyle{vertex} = [shape = circle, minimum size = 4pt, inner sep = 1pt] 
\node[vertex] (G--1) at (0.0, -1) [shape = circle, draw,fill=black] {}; 
\node[vertex] (G-1) at (0.0, 1) [shape = circle, draw,fill=black] {}; 
\draw (G-1) .. controls +(0, -1) and +(0, 1) .. (G--1); 
\end{tikzpicture}
}

\swapnumbers 
\newtheorem{thm}{Theorem}[section]
\newtheorem*{thm*}{Theorem}
\newtheorem{lem}[thm]{Lemma}
\newtheorem*{lem*}{Lemma}
\newtheorem{prop}[thm]{Proposition}
\newtheorem*{prop*}{Proposition}
\newtheorem{cor}[thm]{Corollary}
\newtheorem*{cor*}{Corollary}

\newtheorem*{conj*}{Conjecture}

\theoremstyle{definition}
\newtheorem{defn}[thm]{Definition}
\newtheorem*{defn*}{Definition}
\newtheorem{example}[thm]{Example}
\newtheorem*{example*}{Example}
\newtheorem{rmk}[thm]{Remark}
\newtheorem*{rmk*}{Remark}

\newtheorem*{que*}{Question}
\newtheorem{algo}[thm]{Algorithm}
\newtheorem*{algo*}{Algorithm}

\newcommand{\C}{\mathbb{C}} 
\newcommand{\R}{\mathbb{R}} 
\newcommand{\Q}{\mathbb{Q}} 
\newcommand{\Z}{\mathbb{Z}} 
\newcommand{\N}{\mathbb{N}} 
\newcommand{\K}{\mathbb{K}} 
\newcommand{\UU}{\mathbf{U}} 
\newcommand{\A}{\mathbb{A}} 
\newcommand{\TL}{\operatorname{TL}} 
\newcommand{\End}{\operatorname{End}} 
\newcommand{\Hom}{\operatorname{Hom}} 
\newcommand{\gl}{\mathfrak{gl}} 
\newcommand{\fsl}{\mathfrak{sl}} 
\newcommand{\bil}[2]{\langle #1, #2 \rangle}

\newcommand{\D}{\mathscr{D}} 
\newcommand{\Cat}{\operatorname{LW}} 
\newcommand{\HH}{\mathbf{H}} 
\newcommand{\Sym}{\mathfrak{S}} 

\newcommand{\dbracket}[1]{\llbracket #1 \rrbracket} 

\newcommand{\tr}{\operatorname{tr}} 
\newcommand{\Tr}{\operatorname{Tr}} 
\newcommand{\E}{\mathcal{E}} 
\newcommand{\ZZ}{\mathcal{A}} 

\renewcommand{\labelenumi}{(\alph{enumi})}
\allowdisplaybreaks

\title[Origins of the Temperley--Lieb algebra]%
      {Origins of the Temperley--Lieb algebra: early history}
\author{Stephen Doty}
\email{doty@math.luc.edu}
\address{\parbox{\linewidth}{Department of Mathematics and Statistics,
  Loyola University Chicago,\\ Chicago, IL 60660 USA}}
\author{Anthony Giaquinto}
\email{tonyg@math.luc.edu}
\address{\parbox{\linewidth}{Department of Mathematics and Statistics,
  Loyola University Chicago,\\ Chicago, IL 60660 USA}}
\subjclass{Primary 16T30, 16G99, 17B37}
\keywords{Diagram algebras, Schur--Weyl duality,
  Temperley--Lieb algebras, quantized enveloping algebras}
\begin{document}
\begin{abstract}\noindent
We give an historical survey of some of the original basic algebraic
and combinatorial results on Temperley--Lieb algebras, with a focus on
certain results that have become folklore.
\end{abstract}
\maketitle
\tableofcontents

\section{Introduction}\noindent
The Temperley--Lieb algebra $\TL_n(\delta)$ was was introduced in
\cite{TL} in connection with certain problems in mathematical physics.
It reappeared in the 1980s as a certain von Neumann algebra in the
spectacular work of Vaughan Jones
\cites{Jones:83,Jones:85,Jones:86,Jones,Jones:91} on subfactors and
knots.  Kauffman \cites{K:87,K:88,K:90} (see also \cite{Kauffman})
realized it as a diagram algebra. Birman and Wenzl \cite{Birman-Wenzl}
showed that it is isomorphic to a subalgebra of the Brauer algebra
\cite{Brauer}; this also follows from Kauffman's results.

This paper is an historical survey of the most fundamental algebraic
and combinatorial results on these algebras and their
representations. When writing \cites{DG:PTL,DG:orthog}, we found it
challenging to track down original references for various basic
results that have become folklore. The purpose of this paper is to
document what we found, hopefully aiding subsequent researchers
working in this area. Our focus is somewhat different from that of the
excellent survey article \cite{Ridout-StAubin}. Furthermore, we wish
to draw attention to the book \cite{GHJ}, a reference which seems to
be almost universally ignored by authors writing papers in this area.
Since our focus is on early history, we do not attempt to survey more
recent important developments such as categorification
\cites{BFK,Stroppel,FKS} or the many generalizations of
Temperley--Lieb algebras that exist in abundance in the literature.

In the early papers the ground field was always the complex field
$\C$. The book \cite{GHJ} replaces $\C$ by an arbitrary field
$\K$, which is sometimes assumed to be of characteristic
zero. Many basic properties of Temperley--Lieb algebras are valid even
more generally, where the field $\K$ is replaced by an
arbitrary (unital) commutative ring, and we work in that context
whenever possible.

Most of the results in this survey appeared prior to the turn of the
millennium. A notable exception is the new algorithm (due to Chris
Bowman, but formulated somewhat differently here) discussed in
\S\ref{s:skew}; this may be our only new result. We also include a
short proof of Schur--Weyl duality for the non-semisimple case in
\S\ref{s:SWD}; this has always been implicit in the literature but to
our knowledge has not been spelled out anywhere.

\section{The Temperley--Lieb algebra}\label{s:1}\noindent
In this section $\Bbbk$ is a commutative ring with $1$. The
Temperley--Lieb algebra appeared originally in \cite{TL} in connection
with the Potts model in mathematical physics. For any positive integer
$n$ and any element $\delta$ in the ground ring $\Bbbk$,
$\TL_n(\delta)$ is the unital $\Bbbk$-algebra defined by generators
$e_1, \dots, e_{n-1}$ subject to the relations
\begin{equation}\label{e:TL}
  \begin{aligned}
  e_i^2 = \delta e_i, \quad
  e_i e_j e_i = e_i \text{ if } |i-j|=1, \quad
  e_i e_j = e_j e_i \text{ if } |i-j|>1 .
\end{aligned}
\end{equation}
The unit element $1$ of the algebra is identified with the empty
product of generators.  There is an algebra isomorphism $\TL_n(\delta)
\cong \TL_n(-\delta)$ defined on generators by $e_i \mapsto -e_i$.

\begin{rmk}
The above description by generators and relations is not explicitly
given in \cite{TL}, but can be found, for instance, in the book
\cite{Baxter}.
\end{rmk}

Our first task is to show that $\TL_n = \TL_n(\delta)$ has a finite
spanning set over $\Bbbk$ (hence is finite-dimensional if $\Bbbk$ is a
field). We follow an argument sketched in Jones \cite{Jones:83}.
Words $w$, $w'$ in the generators $e_1, \dots, e_{n-1}$ are
\emph{equivalent} (written as $w_1 \sim w_2$) if they are equal up to
a factor which is a power of $\delta$.  Say that a word $w = e_{i_1}
\cdots e_{i_l}$ in $\TL_n$ is \emph{reduced} if it has minimal
possible length in its equivalence class.


\begin{lem}[Jones' Lemma]
  If $w = e_{i_1} \cdots e_{i_l}$ is a reduced word in $\TL_n$ then $m
  := \max\{i_1, \dots, i_l\}$ occurs only once in the sequence $(i_1,
  \dots, i_l)$.
\end{lem}

\begin{proof}
This is proved by induction on the length. The base case is trivial.
Let $w$ be a reduced word. Suppose for contradiction that $e_m$
appears at least twice in $w$, where $m$ is the maximal index that
appears. Then $w = w_1e_mw_2e_mw_3$. We may assume that $w_2$ does not
contain $e_m$.

If $w_2$ does not contain $e_{m-1}$ then $e_m$ commutes with all the
$e_i$ appearing in $w_2$, so after commuting the rightmost $e_m$ to
the left of $w_2$, the length of $w$ can be shortened using the
equivalence $e_m^2 \sim e_m$. Contradiction.

The remaining case is that $w_2$ contains $e_{m-1}$. Now $w_2$ is
reduced since $w$ is. By the inductive hypothesis, $w_2 = w_4 e_{m-1}
w_5$ where $w_4$, $w_5$ are words on $e_1, \dots, e_{m-2}$. Thus $w_4$
can be commuted to the left and $w_5$ to the right, and the length of
$w$ can once again be shortened using the equivalence $e_m e_{m-1} e_m
\sim e_m$. Contradiction.
\end{proof}

It follows immediately from Jones' Lemma by induction on $n$ that
there are only finitely many reduced words in $\TL_n$, hence that
$\TL_n$ has finite spanning set over $\Bbbk$.

By Jones' Lemma, if $w$ is a reduced word in which $e_m$ is the
generator of maximum index, then by commuting $e_m$ as far to the
right as possible, and commuting subsequent generators of smaller
index as far to the left as possible, we have
\[
w = w' (e_m e_{m-1} \cdots e_{m-l}), \qquad l\ge 0
\]
where $w'$ is a reduced word in which the generator of maximum index
is strictly smaller than $m$. If necessary, reduce $e_{m-1} e_m
e_{m-1}$ to $e_{m-1}$. Induction then leads to the following.

\begin{thm}[Jones' normal form]\label{t:Jones}
  Any reduced word $w$ in $\TL_n$ may be written in the form
  \[
  w = (e_{j_1}e_{j_1-1} \cdots e_{k_1}) (e_{j_2}e_{j_2-1} \cdots
  e_{k_2}) \cdots (e_{j_r}e_{j_r-1} \cdots e_{k_r})
  \]
  where $0 < j_1 < \cdots < j_r < n$, $0 < k_1 < \cdots < k_r < n$,
  and $j_i \ge k_i$ for all $i$. The index $j_r$ is the maximum index
  appearing in $w$.
\end{thm}

\begin{rmk}\label{r:dual-JNF}
By interchanging right and left in the above argument, we obtain a
``dual'' version of the Jones normal form in which the inequalities
are reversed. Details are left to the reader.
\end{rmk}

To each word in Jones normal form as above we may associate a
piecewise linear increasing (planar) path from $(0,0)$ to $(n,n)$:
\begin{multline*}
(0,0) \to (j_1,0) \to (j_1,k_1) \to (j_2,k_1) \to (j_2,k_2) \to \cdots
\\ \to (j_r,k_{r-1}) \to (j_r,k_r) \to (n,k_r) \to (n,n)
\end{multline*}
on the integer lattice $\Z \times \Z$ which does not cross the
diagonal.  Such paths are known as \emph{Dyck paths}.  For instance,
the word $(e_3e_2e_1)(e_4e_3)(e_5e_4)$ in $\TL_6$ corresponds to the
Dyck path
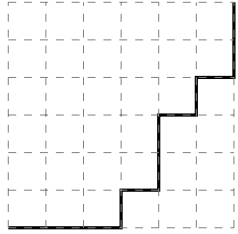
\[
\begin{tikzpicture}[scale = 0.5,thick, baseline={(0,-1ex/2)}] 
  \draw[very thick]
  (0,0)--(3,0)--(3,1)--(4,1)--(4,3)--(5,3)--(5,4)--(6,4)--(6,6);
  \draw[step=1,dashed,gray,very thin] (0,0) grid (6,6);
\end{tikzpicture}
\]
from $(0,0)$ in the lower left corner to $(6,6)$ in the upper right
corner. This map from words to Dyck paths is a bijection because
each such walk is determined by its corner points, and the normal form
of the word can be reconstructed from the coordinates of those points.

To count the number of Dyck paths to $(n,n)$ it is useful to consider
a slightly more general question. For any integer $p$ satisfying $0
\le 2p \le n$, define a \emph{lattice walk} to $(n-p,p)$ to be a
piecewise linear increasing path from $(0,0)$ to $(n-p,p)$ on the
integer lattice $\Z \times \Z$ which does not cross the diagonal. In
particular, Dyck paths are lattice walks to $(n,n)$. Let
\[
\Cat_{n,\,p} =
\begin{minipage}{3.2in}
number of lattice walks from $(0,0)$ to $(n-p,p)$.
\end{minipage}
\]
In this notation, $\Cat_{2n,n}$ gives the number of Dyck paths to
$(n,n)$. Any lattice walk to $(n-p,p)$ must pass through either
$(n-p-1,p)$ or $(n-p,p-1)$, so
\begin{equation}\label{e:recurrence}
  \Cat_{n,\,p} = \Cat_{n-1,\,p} + \Cat_{n-1,\,p-1}
\end{equation}
where $\Cat_{n,-1} = 0$.  We clearly have $\Cat_{n,\,0} = 1$. Also,
$\Cat_{2p-1,\,p} = 0$ as walks are not allowed to cross the
diagonal. With these boundary conditions the recurrence
\eqref{e:recurrence} is easily solved, giving the formula
\begin{equation}\label{e:Cnp}
  \Cat_{n,\,p} = \binom{n}{p}-\binom{n}{p-1}
\end{equation}
where we interpret $\binom{n}{-1}$ as zero, as usual. The set of words
in normal form spans $\TL_n$, and the $n$th Catalan number
\begin{equation}\label{e:upper-bound}
  \Cat_{2n,n} = \binom{2n}{n}-\binom{2n}{n-1} =
  \frac{1}{n+1}\binom{2n}{n}
\end{equation}
gives its cardinality. Linear independence of words in normal form
will be proved in the next section, thus showing that the set of such
words is in fact a basis of $\TL_n$ and thus $\TL_n$ is free as a
$\Bbbk$-module, of rank $\Cat_{2n,n}$.


\section{Diagrammatics}\label{s:dia}\noindent
We continue to work over an arbitrary commutative ring $\Bbbk$, where
$\delta$ is a fixed element of $\Bbbk$.  Following Kauffman, we now
introduce a diagram algebra $\D_n(\delta)$, based on planar diagrams
called $n$-diagrams.  It will turn out that $\D_n(\delta)$ is
isomorphic to $\TL_n(\delta)$.

An $n$-\emph{diagram} is a planar graph, with $2n$ vertices consisting
of $n$ marked points on each of two parallel lines. Each point in the
graph is the endpoint of precisely one edge, and the edges can be
drawn by non-intersecting arcs which lie entirely between the
lines. If we label the vertices along one line by the set $\mathbf{n}
= \{1, \dots, n\}$ and by $\mathbf{n}' = \{1', \dots, n'\}$
correspondingly along the other line, where $\mathbf{n} \cap
\mathbf{n}' = \emptyset$, then we may identify an $n$-diagram $D$ with
a set partition $\{B_1, \dots, B_n\}$ of $\mathbf{n} \sqcup
\mathbf{n}'$ in which each subset (block) has cardinality two. For
example, the $8$-diagram
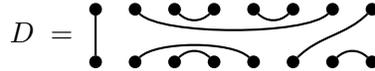
\[
D \;=\;
\begin{tikzpicture}[scale = 0.35,thick, baseline={(0,-1ex/2)}] 
  \tikzstyle{vertex} = [shape = circle, minimum size = 4pt, inner sep = 1pt,
    fill=black] 
\node[vertex] (G--8) at (10.5, -1) [shape = circle, draw] {}; 
\node[vertex] (G--7) at (9.0, -1) [shape = circle, draw] {}; 
\node[vertex] (G--6) at (7.5, -1) [shape = circle, draw] {}; 
\node[vertex] (G-8) at (10.5, 1) [shape = circle, draw] {}; 
\node[vertex] (G--5) at (6.0, -1) [shape = circle, draw] {}; 
\node[vertex] (G--2) at (1.5, -1) [shape = circle, draw] {}; 
\node[vertex] (G--4) at (4.5, -1) [shape = circle, draw] {}; 
\node[vertex] (G--3) at (3.0, -1) [shape = circle, draw] {}; 
\node[vertex] (G--1) at (0.0, -1) [shape = circle, draw] {}; 
\node[vertex] (G-1) at (0.0, 1) [shape = circle, draw] {}; 
\node[vertex] (G-2) at (1.5, 1) [shape = circle, draw] {}; 
\node[vertex] (G-7) at (9.0, 1) [shape = circle, draw] {}; 
\node[vertex] (G-3) at (3.0, 1) [shape = circle, draw] {}; 
\node[vertex] (G-4) at (4.5, 1) [shape = circle, draw] {}; 
\node[vertex] (G-5) at (6.0, 1) [shape = circle, draw] {}; 
\node[vertex] (G-6) at (7.5, 1) [shape = circle, draw] {}; 
\draw[] (G--8) .. controls +(-0.5, 0.5) and +(0.5, 0.5) .. (G--7); 
\draw[] (G-8) .. controls +(-1, -1) and +(1, 1) .. (G--6); 
\draw[] (G--5) .. controls +(-0.8, 0.8) and +(0.8, 0.8) .. (G--2); 
\draw[] (G--4) .. controls +(-0.5, 0.5) and +(0.5, 0.5) .. (G--3); 
\draw[] (G-1) .. controls +(0, -1) and +(0, 1) .. (G--1); 
\draw[] (G-2) .. controls +(1, -1) and +(-1, -1) .. (G-7); 
\draw[] (G-3) .. controls +(0.5, -0.5) and +(-0.5, -0.5) .. (G-4); 
\draw[] (G-5) .. controls +(0.5, -0.5) and +(-0.5, -0.5) .. (G-6); 
\end{tikzpicture}
\]
corresponds to the set partition
\[
\{ \{1,1'\}, \{2,7\}, \{3,4\},
\{5,6\}, \{8,6'\}, \{2',5'\}, \{3',4'\}, \{7',8'\} \}.
\]
In the literature, edges in $n$-diagrams are also called
\emph{strands} or \emph{links}. Some authors refer to edges connecting
two vertices in the top or bottom row as \emph{cups} or \emph{caps},
respectively, and to edges connecting a top vertex to a bottom one as
\emph{through strings} or \emph{propagating strands}.

We define a multiplication on the set of $n$-diagrams as follows. If
$D_1$, $D_2$ are given $n$-diagrams, we stack $D_1$ on top of $D_2$,
identifying the middle lines and their vertices. This results in a
graph with zero or more loops in the middle, and we define
\begin{equation}\label{e:mult-rule}
D_1 D_2 = \delta^L \, D_3
\end{equation}
where $L$ is the number of loops and $D_3$ is the $n$-diagram obtained
by removing the middle data (lines, loops, and vertices). Write
\[
\D_n(\delta) = \Bbbk\text{-linear span of the set of all $n$-diagrams}.
\]
With the multiplication rule given above, $\D_n(\delta)$ is an
associative algebra over $\Bbbk$. The set of $n$-diagrams is a
$\Bbbk$-basis of $\D_n(\delta)$.

Let $\hat{e}_i= \{\{i,i+1\}, \{i',(i+1)'\}\} \sqcup \{ \{j,j'\}\mid j
\ne i, i+1 \}$ ($i = 1, \dots, n-1$) and $\hat{1}= \{ \{i,i'\}\mid i =
1, \dots, n \}$.
Then
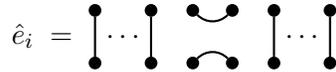
\[
\hat{e}_i \;=\;
\onepic \cdots \onepic\quad \epic \quad\onepic \cdots \onepic
\]
and one checks from the multiplication rule \eqref{e:mult-rule} that
the $\hat{e}_i$ ($1 \le i \le n-1$) satisfy the defining relations
\eqref{e:TL} for $\TL_n(\delta)$. Moreover, $\hat{1} d = d = d \hat{1}$
for all $n$-diagrams $d$. It follows that there is a unique algebra
morphism
\begin{equation}\label{e:morphism}
  \sigma: \TL_n(\delta) \to \D_n(\delta)
\end{equation}
such that $\sigma(1) = \hat{1}$ and $\sigma(e_i) = \hat{e}_i$ for all $i
= 1, \dots, n-1$.

In order to count the number of $n$-diagrams, it is again fruitful to
consider a more general problem: counting the number of half-diagrams.
Cutting a diagram by a line halfway between (and parallel to) its
defining parallel lines divides the diagram into two half-diagrams. We
conventionally reflect the bottom half-diagram across its line of
marked points, so that all half-diagrams are oriented with the links
lying below the line.  A half-diagram coming from an $n$-diagram has
$p$ links (arcs connecting two vertices) and $n-2p$ defects (arcs with
one vertex), where $0 \le 2p \le n$.

\begin{lem}\label{l:half-count}
  $\Cat_{n,p} = $ the number of half-diagrams on $n$ vertices with $p$
  links.
\end{lem}

\begin{proof}
The set of all half-diagrams on $n$ vertices with $p$ links is in
bijection with the set of lattice walks (see \S\ref{s:1}) from $(0,0)$
to $(n-p,p)$. Reading a half-diagram from left to right, a walker
moves up at the $k$-th step if the $k$-th marked point closes a link,
and moves right otherwise. For example, the half-diagram
\[
\begin{tikzpicture}[scale = 0.35,thick, baseline={(0,-1ex/2)}] 
  \tikzstyle{vertex} = [shape = circle, minimum size = 4pt, inner sep = 1pt,
    fill=black] 
\node[vertex] (G-1) at (0.0, 1) [shape = circle, draw] {}; 
\node[vertex] (G-2) at (1.5, 1) [shape = circle, draw] {}; 
\node[vertex] (G-3) at (3.0, 1) [shape = circle, draw] {}; 
\node[vertex] (G-4) at (4.5, 1) [shape = circle, draw] {}; 
\node[vertex] (G-5) at (6.0, 1) [shape = circle, draw] {}; 
\node[vertex] (G-6) at (7.5, 1) [shape = circle, draw] {};
\node[vertex] (G-7) at (9.0, 1) [shape = circle, draw] {};
\node[vertex] (G-8) at (10.5, 1) [shape = circle, draw] {};
\draw[] (G-8) -- (10.5,0); 
\draw[] (G-1) -- (0,0); 
\draw[] (G-2) .. controls +(1, -1) and +(-1, -1) .. (G-7); 
\draw[] (G-3) .. controls +(0.5, -0.5) and +(-0.5, -0.5) .. (G-4); 
\draw[] (G-5) .. controls +(0.5, -0.5) and +(-0.5, -0.5) .. (G-6); 
\end{tikzpicture}
\]
with $8$ vertices and $3$ links corresponds to the lattice walk
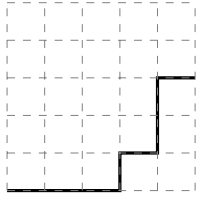
\[
\begin{tikzpicture}[scale = 0.5,thick, baseline={(0,-1ex/2)}] 
\draw[very thick] (0,0)--(3,0)--(3,1)--(4,1)--(4,3)--(5,3);
\draw[step=1,dashed,gray,very thin] (0,0) grid (5,5);
\end{tikzpicture}
\]
from $(0,0)$ to $(5,3) = (n-p,p)$. The half-diagram may be
reconstructed from the lattice walk, so this is a bijection as claimed. 
\end{proof}

Kauffman observed that the set of $n$-diagrams is in a natural
bijection with the set of half-diagrams with $2n$ vertices and $n$
links. This bijection is visualized by the following picture:
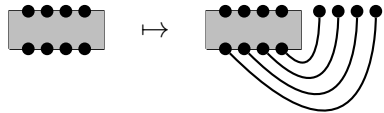
\[
\begin{tikzpicture}[scale = 0.25,thick, baseline={(0,-1ex/2)}]
  \tikzstyle{vertex} = [shape = circle, minimum size = 4pt, inner sep = 1pt,
    fill=black]
  \path[use as bounding box] (0, -1) rectangle (5, 1);
  \draw[thick] (0,-1)--(5,-1)--(5,1)--(0,1)--(0,-1);
  \fill[lightgray] (0,-1) rectangle (5,1);
  \node[vertex] (B-1) at (1, -1) [shape = circle, draw] {};
  \node[vertex] (B-2) at (2, -1) [shape = circle, draw] {};
  \node[vertex] (B-3) at (3, -1) [shape = circle, draw] {};
  \node[vertex] (B-4) at (4, -1) [shape = circle, draw] {};
  \node[vertex] (T-1) at (1, 1) [shape = circle, draw] {};
  \node[vertex] (T-2) at (2, 1) [shape = circle, draw] {};
  \node[vertex] (T-3) at (3, 1) [shape = circle, draw] {};
  \node[vertex] (T-4) at (4, 1) [shape = circle, draw] {};
\end{tikzpicture}
\quad \mapsto \quad
\begin{tikzpicture}[scale = 0.25,thick, baseline={(0,-1ex/2)}]
  \tikzstyle{vertex} = [shape = circle, minimum size = 4pt, inner sep = 1pt,
    fill=black]
  \path[use as bounding box] (0, -4) rectangle (5, 1.5);
  \draw[thick] (0,-1)--(5,-1)--(5,1)--(0,1)--(0,-1);
  \fill[lightgray] (0,-1) rectangle (5,1);
  \node[vertex] (B-1) at (1, -1) [shape = circle, draw] {};
  \node[vertex] (B-2) at (2, -1) [shape = circle, draw] {};
  \node[vertex] (B-3) at (3, -1) [shape = circle, draw] {};
  \node[vertex] (B-4) at (4, -1) [shape = circle, draw] {};
  \node[vertex] (T-1) at (1, 1) [shape = circle, draw] {};
  \node[vertex] (T-2) at (2, 1) [shape = circle, draw] {};
  \node[vertex] (T-3) at (3, 1) [shape = circle, draw] {};
  \node[vertex] (T-4) at (4, 1) [shape = circle, draw] {};
  \node[vertex] (T-5) at (6, 1) [shape = circle, draw] {};
  \node[vertex] (T-6) at (7, 1) [shape = circle, draw] {};
  \node[vertex] (T-7) at (8, 1) [shape = circle, draw] {};
  \node[vertex] (T-8) at (9, 1) [shape = circle, draw] {};
  \draw[] (B-4) .. controls +(2, -2) and +(0,-0.5) .. (T-5);
  \draw[] (B-3) .. controls +(4, -4) and +(0,-0.5) .. (T-6);
  \draw[] (B-2) .. controls +(6, -6) and +(0,-0.5) .. (T-7);
  \draw[] (B-1) .. controls +(8, -8) and +(0,-0.5) .. (T-8);
\end{tikzpicture}
\]
In other words, draw an $n$-diagram in a rectangle, and rotate its
bottom edge through an angle of $180^\circ$, with its vertex at the
upper right corner of the rectangle. Edges are stretched accordingly
to maintain the planarity. For example,
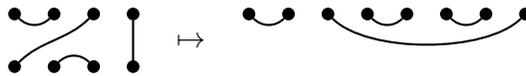
\[
\begin{tikzpicture}[scale = 0.35,thick, baseline={(0,-1ex/2)}] 
  \tikzstyle{vertex} = [shape = circle, minimum size = 4pt, inner sep = 1pt,
    fill=black] 
\node[vertex] (G--4) at (4.5, -1) [shape = circle, draw] {}; 
\node[vertex] (G-4) at (4.5, 1) [shape = circle, draw] {}; 
\node[vertex] (G--3) at (3.0, -1) [shape = circle, draw] {}; 
\node[vertex] (G--2) at (1.5, -1) [shape = circle, draw] {}; 
\node[vertex] (G--1) at (0.0, -1) [shape = circle, draw] {}; 
\node[vertex] (G-3) at (3.0, 1) [shape = circle, draw] {}; 
\node[vertex] (G-1) at (0.0, 1) [shape = circle, draw] {}; 
\node[vertex] (G-2) at (1.5, 1) [shape = circle, draw] {}; 
\draw[] (G-4) .. controls +(0, -1) and +(0, 1) .. (G--4); 
\draw[] (G--3) .. controls +(-0.5, 0.5) and +(0.5, 0.5) .. (G--2); 
\draw[] (G-3) .. controls +(-1, -1) and +(1, 1) .. (G--1); 
\draw[] (G-1) .. controls +(0.5, -0.5) and +(-0.5, -0.5) .. (G-2); 
\end{tikzpicture}
\quad \mapsto \quad
\begin{tikzpicture}[scale = 0.35,thick, baseline={(0,-1ex/2)}] 
  \tikzstyle{vertex} = [shape = circle, minimum size = 4pt, inner sep = 1pt,
    fill=black] 
\node[vertex] (G-4) at (4.5, 1) [shape = circle, draw] {}; 
\node[vertex] (G-3) at (3.0, 1) [shape = circle, draw] {}; 
\node[vertex] (G-1) at (0.0, 1) [shape = circle, draw] {}; 
\node[vertex] (G-2) at (1.5, 1) [shape = circle, draw] {};
\node[vertex] (T-5) at (6,1) [shape = circle, draw] {};
\node[vertex] (T-6) at (7.5,1) [shape = circle, draw] {};
\node[vertex] (T-7) at (9,1) [shape = circle, draw] {};
\node[vertex] (T-8) at (10.5,1) [shape = circle, draw] {};
\draw[] (G-1) .. controls +(0.5, -0.5) and +(-0.5, -0.5) .. (G-2);
\draw[] (G-4) .. controls +(0.5, -0.5) and +(-0.5, -0.5) .. (T-5);
\draw[] (T-6) .. controls +(0.5, -0.5) and +(-0.5, -0.5) .. (T-7);
\draw[] (G-3) .. controls +(1.5, -1.5) and +(-1.5, -1.5) .. (T-8);
\end{tikzpicture}
\]
This process of mapping $n$-diagrams to half-diagrams (with $2n$
vertices and $n$ links) is clearly reversible, hence defines a
bijection.  Thus, $\Cat_{2n,n}$ counts this number, and 
\begin{equation}
  \text{rank}_{\Bbbk} \D_n(\delta) = \Cat_{2n,n}.
\end{equation}
This is again the $n$th Catalan number. It appeared already in
\eqref{e:upper-bound}, as the number of words in $\TL_n(\delta)$ in
Jones normal form.

Kauffman \cite{K:90} found an algorithm that expresses a given
$n$-diagram by a reduced expression as a product of the
$\hat{e}_i$.

Start by drawing the diagram inside its enclosing rectangle. As
strands do not cross, the diagram partitions the rectangle into a
disjoint union of open regions.  Number the regions according to the
natural ``reading'' order: left to right within top to bottom.  For
example, the following picture gives the ordering
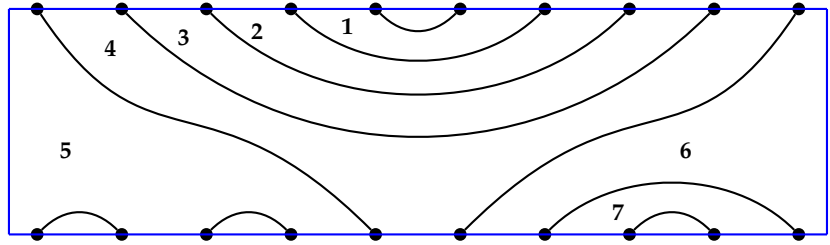
\[
\begin{tikzpicture}[scale = 0.75,thick, baseline={(0,-1ex/2)}] 
\tikzstyle{vertex} = [shape=circle, minimum size=4pt, inner sep=1pt,fill=black] 
\node[vertex] (G--10) at (13.5, -2) [shape = circle, draw] {}; 
\node[vertex] (G--7) at (9.0, -2) [shape = circle, draw] {}; 
\node[vertex] (G--9) at (12.0, -2) [shape = circle, draw] {}; 
\node[vertex] (G--8) at (10.5, -2) [shape = circle, draw] {}; 
\node[vertex] (G--6) at (7.5, -2) [shape = circle, draw] {}; 
\node[vertex] (G-10) at (13.5, 2) [shape = circle, draw] {}; 
\node[vertex] (G--5) at (6.0, -2) [shape = circle, draw] {}; 
\node[vertex] (G-1) at (0.0, 2) [shape = circle, draw] {}; 
\node[vertex] (G--4) at (4.5, -2) [shape = circle, draw] {}; 
\node[vertex] (G--3) at (3.0, -2) [shape = circle, draw] {}; 
\node[vertex] (G--2) at (1.5, -2) [shape = circle, draw] {}; 
\node[vertex] (G--1) at (0.0, -2) [shape = circle, draw] {}; 
\node[vertex] (G-2) at (1.5, 2) [shape = circle, draw] {}; 
\node[vertex] (G-9) at (12.0, 2) [shape = circle, draw] {}; 
\node[vertex] (G-3) at (3.0, 2) [shape = circle, draw] {}; 
\node[vertex] (G-8) at (10.5, 2) [shape = circle, draw] {}; 
\node[vertex] (G-4) at (4.5, 2) [shape = circle, draw] {}; 
\node[vertex] (G-7) at (9.0, 2) [shape = circle, draw] {}; 
\node[vertex] (G-5) at (6.0, 2) [shape = circle, draw] {}; 
\node[vertex] (G-6) at (7.5, 2) [shape = circle, draw] {};
\draw[blue] (-0.5,2) -- (14,2);
\draw[blue] (-0.5,-2) -- (14,-2);
\draw[blue] (-0.5,2) -- (-0.5,-2);
\draw[blue] (14,2) -- (14,-2);
\draw[name path=i] (G--10) .. controls +(-1.2, 1.2) and +(1.2, 1.2) .. (G--7); 
\draw[name path=j] (G--9) .. controls +(-0.5, 0.5) and +(0.5, 0.5) .. (G--8); 
\draw[name path=h] (G-10) .. controls +(-2, -3) and +(3, 3) .. (G--6); 
\draw[name path=b] (G-1) .. controls +(2, -3) and +(-3, 3) .. (G--5); 
\draw[name path=d] (G--4) .. controls +(-0.5, 0.5) and +(0.5, 0.5) .. (G--3); 
\draw[name path=a] (G--2) .. controls +(-0.5, 0.5) and +(0.5, 0.5) .. (G--1); 
\draw[name path=c] (G-2) .. controls +(3, -3) and +(-3, -3) .. (G-9); 
\draw[name path=e] (G-3) .. controls +(2, -2) and +(-2, -2) .. (G-8); 
\draw[name path=f] (G-4) .. controls +(1.2, -1.2) and +(-1.2, -1.2) .. (G-7); 
\draw[name path=g] (G-5) .. controls +(0.5, -0.5) and +(-0.5, -0.5) .. (G-6);

\node at (5.5,1.7) {\footnotesize\bf1};
\node at (3.9,1.6) {\footnotesize\bf2};
\node at (2.6,1.5) {\footnotesize\bf3};
\node at (1.3,1.3) {\footnotesize\bf4};
\node at (0.5,-0.5) {\footnotesize\bf5};
\node at (11.5,-0.5) {\footnotesize\bf6};
\node at (10.3,-1.6) {\footnotesize\bf7};
\end{tikzpicture}
\]
for its enclosed $10$-diagram, except that four regions, one at the
top and three along the bottom, have not been numbered. The unnumbered
regions (for which the entire upper or lower boundary is a segment of
the rectangle) don't matter for Kauffman's algorithm.

Next, let $\ell_i$ be the vertical line bisecting the rectangle with
vertices at the nodes $i$, $i+1$, $i'$, $(i+1)'$. This line always
crosses an even number (possibly zero) of strands in the
diagram. Connect the intersection points in consecutive pairs by a
dashed line segment along $\ell_i$. Do this for each $i = 1, \dots,
n-1$. Here is the above diagram with its connections.
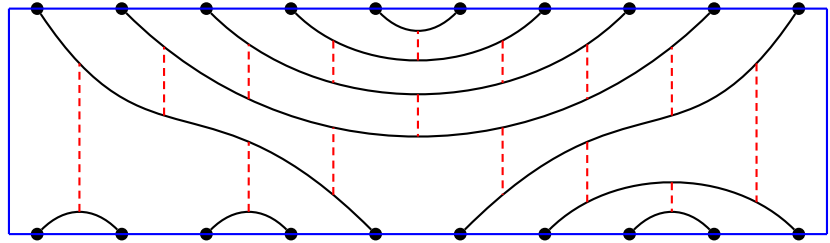
\[
\begin{tikzpicture}[scale = 0.75,thick, baseline={(0,-1ex/2)}] 
\tikzstyle{vertex} = [shape=circle, minimum size=4pt, inner sep=1pt,fill=black] 
\node[vertex] (G--10) at (13.5, -2) [shape = circle, draw] {}; 
\node[vertex] (G--7) at (9.0, -2) [shape = circle, draw] {}; 
\node[vertex] (G--9) at (12.0, -2) [shape = circle, draw] {}; 
\node[vertex] (G--8) at (10.5, -2) [shape = circle, draw] {}; 
\node[vertex] (G--6) at (7.5, -2) [shape = circle, draw] {}; 
\node[vertex] (G-10) at (13.5, 2) [shape = circle, draw] {}; 
\node[vertex] (G--5) at (6.0, -2) [shape = circle, draw] {}; 
\node[vertex] (G-1) at (0.0, 2) [shape = circle, draw] {}; 
\node[vertex] (G--4) at (4.5, -2) [shape = circle, draw] {}; 
\node[vertex] (G--3) at (3.0, -2) [shape = circle, draw] {}; 
\node[vertex] (G--2) at (1.5, -2) [shape = circle, draw] {}; 
\node[vertex] (G--1) at (0.0, -2) [shape = circle, draw] {}; 
\node[vertex] (G-2) at (1.5, 2) [shape = circle, draw] {}; 
\node[vertex] (G-9) at (12.0, 2) [shape = circle, draw] {}; 
\node[vertex] (G-3) at (3.0, 2) [shape = circle, draw] {}; 
\node[vertex] (G-8) at (10.5, 2) [shape = circle, draw] {}; 
\node[vertex] (G-4) at (4.5, 2) [shape = circle, draw] {}; 
\node[vertex] (G-7) at (9.0, 2) [shape = circle, draw] {}; 
\node[vertex] (G-5) at (6.0, 2) [shape = circle, draw] {}; 
\node[vertex] (G-6) at (7.5, 2) [shape = circle, draw] {};
\draw[blue] (-0.5,2) -- (14,2);
\draw[blue] (-0.5,-2) -- (14,-2);
\draw[blue] (-0.5,2) -- (-0.5,-2);
\draw[blue] (14,2) -- (14,-2);
\draw[name path=i] (G--10) .. controls +(-1.2, 1.2) and +(1.2, 1.2) .. (G--7); 
\draw[name path=j] (G--9) .. controls +(-0.5, 0.5) and +(0.5, 0.5) .. (G--8); 
\draw[name path=h] (G-10) .. controls +(-2, -3) and +(3, 3) .. (G--6); 
\draw[name path=b] (G-1) .. controls +(2, -3) and +(-3, 3) .. (G--5); 
\draw[name path=d] (G--4) .. controls +(-0.5, 0.5) and +(0.5, 0.5) .. (G--3); 
\draw[name path=a] (G--2) .. controls +(-0.5, 0.5) and +(0.5, 0.5) .. (G--1); 
\draw[name path=c] (G-2) .. controls +(3, -3) and +(-3, -3) .. (G-9); 
\draw[name path=e] (G-3) .. controls +(2, -2) and +(-2, -2) .. (G-8); 
\draw[name path=f] (G-4) .. controls +(1.2, -1.2) and +(-1.2, -1.2) .. (G-7); 
\draw[name path=g] (G-5) .. controls +(0.5, -0.5) and +(-0.5, -0.5) .. (G-6);
\path[name path=v1] (0.75,-2)--(0.75,2);
\path[name path=v2] (2.25,-2)--(2.25,2);
\path[name path=v3] (3.75,-2)--(3.75,2);
\path[name path=v4] (5.25,-2)--(5.25,2);
\path[name path=v5] (6.75,-2)--(6.75,2);
\path[name path=v6] (8.25,-2)--(8.25,2);
\path[name path=v7] (9.75,-2)--(9.75,2);
\path[name path=v8] (11.25,-2)--(11.25,2);
\path[name path=v9] (12.75,-2)--(12.75,2);

\path[name intersections={of=a and v1,by=a1}];
\path[name intersections={of=b and v1,by=b1}];
\draw[red,densely dashed] (a1)--(b1);

\path[name intersections={of=b and v2,by=b2}];
\path[name intersections={of=c and v2,by=c2}];
\draw[red,densely dashed] (b2)--(c2);

\path[name intersections={of=d and v3,by=d3}];
\path[name intersections={of=b and v3,by=b3}];
\draw[red,densely dashed] (d3)--(b3);

\path[name intersections={of=c and v3,by=c3}];
\path[name intersections={of=e and v3,by=e3}];
\draw[red,densely dashed] (c3)--(e3);

\path[name intersections={of=f and v4,by=f4}];
\path[name intersections={of=e and v4,by=e4}];
\draw[red,densely dashed] (f4)--(e4);

\path[name intersections={of=b and v4,by=b4}];
\path[name intersections={of=c and v4,by=c4}];
\draw[red,densely dashed] (b4)--(c4);

\path[name intersections={of=f and v5,by=f5}];
\path[name intersections={of=g and v5,by=g5}];
\draw[red,densely dashed] (f5)--(g5);

\path[name intersections={of=e and v5,by=e5}];
\path[name intersections={of=c and v5,by=c5}];
\draw[red,densely dashed] (e5)--(c5);

\path[name intersections={of=f and v6,by=f6}];
\path[name intersections={of=e and v6,by=e6}];
\draw[red,densely dashed] (f6)--(e6);

\path[name intersections={of=c and v6,by=c6}];
\path[name intersections={of=h and v6,by=h6}];
\draw[red,densely dashed] (c6)--(h6);

\path[name intersections={of=e and v7,by=e7}];
\path[name intersections={of=c and v7,by=c7}];
\draw[red,densely dashed] (e7)--(c7);

\path[name intersections={of=h and v7,by=h7}];
\path[name intersections={of=i and v7,by=i7}];
\draw[red,densely dashed] (h7)--(i7);

\path[name intersections={of=h and v8,by=h8}];
\path[name intersections={of=c and v8,by=c8}];
\draw[red,densely dashed] (h8)--(c8);

\path[name intersections={of=i and v8,by=i8}];
\path[name intersections={of=j and v8,by=j8}];
\draw[red,densely dashed] (i8)--(j8);

\path[name intersections={of=h and v9,by=h9}];
\path[name intersections={of=i and v9,by=i9}];
\draw[red,densely dashed] (h9)--(i9);
\end{tikzpicture}
\]
Label each connecting segment on $\ell_i$ by $\hat{e}_i$.  Each
numbered region $k$ has an associated word $w_k(d)$ obtained as the
product of its connection labels in order from left to right. Then
\[
w(d) := w_1(d) \cdots w_r(d) \qquad(\text{$r$ is the number of
  numbered regions})
\]
is a reduced word corresponding to the given diagram $d$. In the
example depicted above, the reduced word is
\[
w(d) = (\hat{e}_5)(\hat{e}_4\hat{e}_6)(\hat{e}_3\hat{e}_5\hat{e}_7)
(\hat{e}_2\hat{e}_4\hat{e}_6\hat{e}_8)(\hat{e}_1\hat{e}_3)
(\hat{e}_7\hat{e}_9)(\hat{e}_8) .
\]
Although not needed, it is easy to apply commutation relations to
rewrite this in its Jones normal form, which in this case is
\[
(\hat{e}_5\hat{e}_4\hat{e}_3\hat{e}_2\hat{e}_1)
(\hat{e}_6\hat{e}_5\hat{e}_4\hat{e}_3)(\hat{e}_7\hat{e}_6)
(\hat{e}_8\hat{e}_7)(\hat{e}_9\hat{e}_8).
\]
That the word $w(d)$ constructs the original diagram $d$ is shown in
Figure~\ref{fig1}. In that figure, the individual diagrams are the
$w_j(d)$, each of which corresponds to a product of generators.
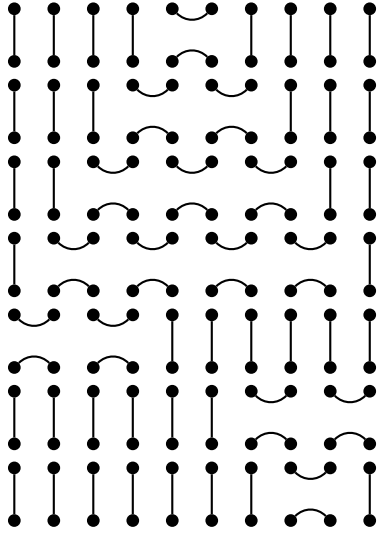
\begin{figure}[ht]
\begin{align*}
\begin{tikzpicture}[scale = 0.35,thick, baseline={(0,-1ex/2)}] 
  \tikzstyle{vertex} = [shape = circle, minimum size = 4pt,
    inner sep=1pt,fill=black] 
\node[vertex] (G--10) at (13.5, -1) [shape = circle, draw] {}; 
\node[vertex] (G-10) at (13.5, 1) [shape = circle, draw] {}; 
\node[vertex] (G--9) at (12.0, -1) [shape = circle, draw] {}; 
\node[vertex] (G-9) at (12.0, 1) [shape = circle, draw] {}; 
\node[vertex] (G--8) at (10.5, -1) [shape = circle, draw] {}; 
\node[vertex] (G-8) at (10.5, 1) [shape = circle, draw] {}; 
\node[vertex] (G--7) at (9.0, -1) [shape = circle, draw] {}; 
\node[vertex] (G-7) at (9.0, 1) [shape = circle, draw] {}; 
\node[vertex] (G--6) at (7.5, -1) [shape = circle, draw] {}; 
\node[vertex] (G--5) at (6.0, -1) [shape = circle, draw] {}; 
\node[vertex] (G--4) at (4.5, -1) [shape = circle, draw] {}; 
\node[vertex] (G-4) at (4.5, 1) [shape = circle, draw] {}; 
\node[vertex] (G--3) at (3.0, -1) [shape = circle, draw] {}; 
\node[vertex] (G-3) at (3.0, 1) [shape = circle, draw] {}; 
\node[vertex] (G--2) at (1.5, -1) [shape = circle, draw] {}; 
\node[vertex] (G-2) at (1.5, 1) [shape = circle, draw] {}; 
\node[vertex] (G--1) at (0.0, -1) [shape = circle, draw] {}; 
\node[vertex] (G-1) at (0.0, 1) [shape = circle, draw] {}; 
\node[vertex] (G-5) at (6.0, 1) [shape = circle, draw] {}; 
\node[vertex] (G-6) at (7.5, 1) [shape = circle, draw] {}; 
\draw[] (G-10) .. controls +(0, -1) and +(0, 1) .. (G--10); 
\draw[] (G-9) .. controls +(0, -1) and +(0, 1) .. (G--9); 
\draw[] (G-8) .. controls +(0, -1) and +(0, 1) .. (G--8); 
\draw[] (G-7) .. controls +(0, -1) and +(0, 1) .. (G--7); 
\draw[] (G--6) .. controls +(-0.5, 0.5) and +(0.5, 0.5) .. (G--5); 
\draw[] (G-4) .. controls +(0, -1) and +(0, 1) .. (G--4); 
\draw[] (G-3) .. controls +(0, -1) and +(0, 1) .. (G--3); 
\draw[] (G-2) .. controls +(0, -1) and +(0, 1) .. (G--2); 
\draw[] (G-1) .. controls +(0, -1) and +(0, 1) .. (G--1); 
\draw[] (G-5) .. controls +(0.5, -0.5) and +(-0.5, -0.5) .. (G-6); 
\end{tikzpicture}
\\
\begin{tikzpicture}[scale = 0.35,thick, baseline={(0,-1ex/2)}] 
\tikzstyle{vertex} = [shape = circle, minimum size = 4pt, inner sep=1pt,fill=black] 
\node[vertex] (G--10) at (13.5, -1) [shape = circle, draw] {}; 
\node[vertex] (G-10) at (13.5, 1) [shape = circle, draw] {}; 
\node[vertex] (G--9) at (12.0, -1) [shape = circle, draw] {}; 
\node[vertex] (G-9) at (12.0, 1) [shape = circle, draw] {}; 
\node[vertex] (G--8) at (10.5, -1) [shape = circle, draw] {}; 
\node[vertex] (G-8) at (10.5, 1) [shape = circle, draw] {}; 
\node[vertex] (G--7) at (9.0, -1) [shape = circle, draw] {}; 
\node[vertex] (G--6) at (7.5, -1) [shape = circle, draw] {}; 
\node[vertex] (G--5) at (6.0, -1) [shape = circle, draw] {}; 
\node[vertex] (G--4) at (4.5, -1) [shape = circle, draw] {}; 
\node[vertex] (G--3) at (3.0, -1) [shape = circle, draw] {}; 
\node[vertex] (G-3) at (3.0, 1) [shape = circle, draw] {}; 
\node[vertex] (G--2) at (1.5, -1) [shape = circle, draw] {}; 
\node[vertex] (G-2) at (1.5, 1) [shape = circle, draw] {}; 
\node[vertex] (G--1) at (0.0, -1) [shape = circle, draw] {}; 
\node[vertex] (G-1) at (0.0, 1) [shape = circle, draw] {}; 
\node[vertex] (G-4) at (4.5, 1) [shape = circle, draw] {}; 
\node[vertex] (G-5) at (6.0, 1) [shape = circle, draw] {}; 
\node[vertex] (G-6) at (7.5, 1) [shape = circle, draw] {}; 
\node[vertex] (G-7) at (9.0, 1) [shape = circle, draw] {}; 
\draw[] (G-10) .. controls +(0, -1) and +(0, 1) .. (G--10); 
\draw[] (G-9) .. controls +(0, -1) and +(0, 1) .. (G--9); 
\draw[] (G-8) .. controls +(0, -1) and +(0, 1) .. (G--8); 
\draw[] (G--7) .. controls +(-0.5, 0.5) and +(0.5, 0.5) .. (G--6); 
\draw[] (G--5) .. controls +(-0.5, 0.5) and +(0.5, 0.5) .. (G--4); 
\draw[] (G-3) .. controls +(0, -1) and +(0, 1) .. (G--3); 
\draw[] (G-2) .. controls +(0, -1) and +(0, 1) .. (G--2); 
\draw[] (G-1) .. controls +(0, -1) and +(0, 1) .. (G--1); 
\draw[] (G-4) .. controls +(0.5, -0.5) and +(-0.5, -0.5) .. (G-5); 
\draw[] (G-6) .. controls +(0.5, -0.5) and +(-0.5, -0.5) .. (G-7); 
\end{tikzpicture}
\\
\begin{tikzpicture}[scale = 0.35,thick, baseline={(0,-1ex/2)}] 
\tikzstyle{vertex} = [shape = circle, minimum size = 4pt, inner sep=1pt,fill=black] 
\node[vertex] (G--10) at (13.5, -1) [shape = circle, draw] {}; 
\node[vertex] (G-10) at (13.5, 1) [shape = circle, draw] {}; 
\node[vertex] (G--9) at (12.0, -1) [shape = circle, draw] {}; 
\node[vertex] (G-9) at (12.0, 1) [shape = circle, draw] {}; 
\node[vertex] (G--8) at (10.5, -1) [shape = circle, draw] {}; 
\node[vertex] (G--7) at (9.0, -1) [shape = circle, draw] {}; 
\node[vertex] (G--6) at (7.5, -1) [shape = circle, draw] {}; 
\node[vertex] (G--5) at (6.0, -1) [shape = circle, draw] {}; 
\node[vertex] (G--4) at (4.5, -1) [shape = circle, draw] {}; 
\node[vertex] (G--3) at (3.0, -1) [shape = circle, draw] {}; 
\node[vertex] (G--2) at (1.5, -1) [shape = circle, draw] {}; 
\node[vertex] (G-2) at (1.5, 1) [shape = circle, draw] {}; 
\node[vertex] (G--1) at (0.0, -1) [shape = circle, draw] {}; 
\node[vertex] (G-1) at (0.0, 1) [shape = circle, draw] {}; 
\node[vertex] (G-3) at (3.0, 1) [shape = circle, draw] {}; 
\node[vertex] (G-4) at (4.5, 1) [shape = circle, draw] {}; 
\node[vertex] (G-5) at (6.0, 1) [shape = circle, draw] {}; 
\node[vertex] (G-6) at (7.5, 1) [shape = circle, draw] {}; 
\node[vertex] (G-7) at (9.0, 1) [shape = circle, draw] {}; 
\node[vertex] (G-8) at (10.5, 1) [shape = circle, draw] {}; 
\draw[] (G-10) .. controls +(0, -1) and +(0, 1) .. (G--10); 
\draw[] (G-9) .. controls +(0, -1) and +(0, 1) .. (G--9); 
\draw[] (G--8) .. controls +(-0.5, 0.5) and +(0.5, 0.5) .. (G--7); 
\draw[] (G--6) .. controls +(-0.5, 0.5) and +(0.5, 0.5) .. (G--5); 
\draw[] (G--4) .. controls +(-0.5, 0.5) and +(0.5, 0.5) .. (G--3); 
\draw[] (G-2) .. controls +(0, -1) and +(0, 1) .. (G--2); 
\draw[] (G-1) .. controls +(0, -1) and +(0, 1) .. (G--1); 
\draw[] (G-3) .. controls +(0.5, -0.5) and +(-0.5, -0.5) .. (G-4); 
\draw[] (G-5) .. controls +(0.5, -0.5) and +(-0.5, -0.5) .. (G-6); 
\draw[] (G-7) .. controls +(0.5, -0.5) and +(-0.5, -0.5) .. (G-8); 
\end{tikzpicture}
\\
\begin{tikzpicture}[scale = 0.35,thick, baseline={(0,-1ex/2)}] 
\tikzstyle{vertex} = [shape = circle, minimum size = 4pt, inner sep=1pt,fill=black] 
\node[vertex] (G--10) at (13.5, -1) [shape = circle, draw] {}; 
\node[vertex] (G-10) at (13.5, 1) [shape = circle, draw] {}; 
\node[vertex] (G--9) at (12.0, -1) [shape = circle, draw] {}; 
\node[vertex] (G--8) at (10.5, -1) [shape = circle, draw] {}; 
\node[vertex] (G--7) at (9.0, -1) [shape = circle, draw] {}; 
\node[vertex] (G--6) at (7.5, -1) [shape = circle, draw] {}; 
\node[vertex] (G--5) at (6.0, -1) [shape = circle, draw] {}; 
\node[vertex] (G--4) at (4.5, -1) [shape = circle, draw] {}; 
\node[vertex] (G--3) at (3.0, -1) [shape = circle, draw] {}; 
\node[vertex] (G--2) at (1.5, -1) [shape = circle, draw] {}; 
\node[vertex] (G--1) at (0.0, -1) [shape = circle, draw] {}; 
\node[vertex] (G-1) at (0.0, 1) [shape = circle, draw] {}; 
\node[vertex] (G-2) at (1.5, 1) [shape = circle, draw] {}; 
\node[vertex] (G-3) at (3.0, 1) [shape = circle, draw] {}; 
\node[vertex] (G-4) at (4.5, 1) [shape = circle, draw] {}; 
\node[vertex] (G-5) at (6.0, 1) [shape = circle, draw] {}; 
\node[vertex] (G-6) at (7.5, 1) [shape = circle, draw] {}; 
\node[vertex] (G-7) at (9.0, 1) [shape = circle, draw] {}; 
\node[vertex] (G-8) at (10.5, 1) [shape = circle, draw] {}; 
\node[vertex] (G-9) at (12.0, 1) [shape = circle, draw] {}; 
\draw[] (G-10) .. controls +(0, -1) and +(0, 1) .. (G--10); 
\draw[] (G--9) .. controls +(-0.5, 0.5) and +(0.5, 0.5) .. (G--8); 
\draw[] (G--7) .. controls +(-0.5, 0.5) and +(0.5, 0.5) .. (G--6); 
\draw[] (G--5) .. controls +(-0.5, 0.5) and +(0.5, 0.5) .. (G--4); 
\draw[] (G--3) .. controls +(-0.5, 0.5) and +(0.5, 0.5) .. (G--2); 
\draw[] (G-1) .. controls +(0, -1) and +(0, 1) .. (G--1); 
\draw[] (G-2) .. controls +(0.5, -0.5) and +(-0.5, -0.5) .. (G-3); 
\draw[] (G-4) .. controls +(0.5, -0.5) and +(-0.5, -0.5) .. (G-5); 
\draw[] (G-6) .. controls +(0.5, -0.5) and +(-0.5, -0.5) .. (G-7); 
\draw[] (G-8) .. controls +(0.5, -0.5) and +(-0.5, -0.5) .. (G-9); 
\end{tikzpicture}
\\
\begin{tikzpicture}[scale = 0.35,thick, baseline={(0,-1ex/2)}] 
\tikzstyle{vertex} = [shape = circle, minimum size = 4pt, inner sep=1pt,fill=black] 
\node[vertex] (G--10) at (13.5, -1) [shape = circle, draw] {}; 
\node[vertex] (G-10) at (13.5, 1) [shape = circle, draw] {}; 
\node[vertex] (G--9) at (12.0, -1) [shape = circle, draw] {}; 
\node[vertex] (G-9) at (12.0, 1) [shape = circle, draw] {}; 
\node[vertex] (G--8) at (10.5, -1) [shape = circle, draw] {}; 
\node[vertex] (G-8) at (10.5, 1) [shape = circle, draw] {}; 
\node[vertex] (G--7) at (9.0, -1) [shape = circle, draw] {}; 
\node[vertex] (G-7) at (9.0, 1) [shape = circle, draw] {}; 
\node[vertex] (G--6) at (7.5, -1) [shape = circle, draw] {}; 
\node[vertex] (G-6) at (7.5, 1) [shape = circle, draw] {}; 
\node[vertex] (G--5) at (6.0, -1) [shape = circle, draw] {}; 
\node[vertex] (G-5) at (6.0, 1) [shape = circle, draw] {}; 
\node[vertex] (G--4) at (4.5, -1) [shape = circle, draw] {}; 
\node[vertex] (G--3) at (3.0, -1) [shape = circle, draw] {}; 
\node[vertex] (G--2) at (1.5, -1) [shape = circle, draw] {}; 
\node[vertex] (G--1) at (0.0, -1) [shape = circle, draw] {}; 
\node[vertex] (G-1) at (0.0, 1) [shape = circle, draw] {}; 
\node[vertex] (G-2) at (1.5, 1) [shape = circle, draw] {}; 
\node[vertex] (G-3) at (3.0, 1) [shape = circle, draw] {}; 
\node[vertex] (G-4) at (4.5, 1) [shape = circle, draw] {}; 
\draw[] (G-10) .. controls +(0, -1) and +(0, 1) .. (G--10); 
\draw[] (G-9) .. controls +(0, -1) and +(0, 1) .. (G--9); 
\draw[] (G-8) .. controls +(0, -1) and +(0, 1) .. (G--8); 
\draw[] (G-7) .. controls +(0, -1) and +(0, 1) .. (G--7); 
\draw[] (G-6) .. controls +(0, -1) and +(0, 1) .. (G--6); 
\draw[] (G-5) .. controls +(0, -1) and +(0, 1) .. (G--5); 
\draw[] (G--4) .. controls +(-0.5, 0.5) and +(0.5, 0.5) .. (G--3); 
\draw[] (G--2) .. controls +(-0.5, 0.5) and +(0.5, 0.5) .. (G--1); 
\draw[] (G-1) .. controls +(0.5, -0.5) and +(-0.5, -0.5) .. (G-2); 
\draw[] (G-3) .. controls +(0.5, -0.5) and +(-0.5, -0.5) .. (G-4); 
\end{tikzpicture}
\\
\begin{tikzpicture}[scale = 0.35,thick, baseline={(0,-1ex/2)}] 
\tikzstyle{vertex} = [shape = circle, minimum size = 4pt, inner sep=1pt,fill=black] 
\node[vertex] (G--10) at (13.5, -1) [shape = circle, draw] {}; 
\node[vertex] (G--9) at (12.0, -1) [shape = circle, draw] {}; 
\node[vertex] (G--8) at (10.5, -1) [shape = circle, draw] {}; 
\node[vertex] (G--7) at (9.0, -1) [shape = circle, draw] {}; 
\node[vertex] (G--6) at (7.5, -1) [shape = circle, draw] {}; 
\node[vertex] (G-6) at (7.5, 1) [shape = circle, draw] {}; 
\node[vertex] (G--5) at (6.0, -1) [shape = circle, draw] {}; 
\node[vertex] (G-5) at (6.0, 1) [shape = circle, draw] {}; 
\node[vertex] (G--4) at (4.5, -1) [shape = circle, draw] {}; 
\node[vertex] (G-4) at (4.5, 1) [shape = circle, draw] {}; 
\node[vertex] (G--3) at (3.0, -1) [shape = circle, draw] {}; 
\node[vertex] (G-3) at (3.0, 1) [shape = circle, draw] {}; 
\node[vertex] (G--2) at (1.5, -1) [shape = circle, draw] {}; 
\node[vertex] (G-2) at (1.5, 1) [shape = circle, draw] {}; 
\node[vertex] (G--1) at (0.0, -1) [shape = circle, draw] {}; 
\node[vertex] (G-1) at (0.0, 1) [shape = circle, draw] {}; 
\node[vertex] (G-7) at (9.0, 1) [shape = circle, draw] {}; 
\node[vertex] (G-8) at (10.5, 1) [shape = circle, draw] {}; 
\node[vertex] (G-9) at (12.0, 1) [shape = circle, draw] {}; 
\node[vertex] (G-10) at (13.5, 1) [shape = circle, draw] {}; 
\draw[] (G--10) .. controls +(-0.5, 0.5) and +(0.5, 0.5) .. (G--9); 
\draw[] (G--8) .. controls +(-0.5, 0.5) and +(0.5, 0.5) .. (G--7); 
\draw[] (G-6) .. controls +(0, -1) and +(0, 1) .. (G--6); 
\draw[] (G-5) .. controls +(0, -1) and +(0, 1) .. (G--5); 
\draw[] (G-4) .. controls +(0, -1) and +(0, 1) .. (G--4); 
\draw[] (G-3) .. controls +(0, -1) and +(0, 1) .. (G--3); 
\draw[] (G-2) .. controls +(0, -1) and +(0, 1) .. (G--2); 
\draw[] (G-1) .. controls +(0, -1) and +(0, 1) .. (G--1); 
\draw[] (G-7) .. controls +(0.5, -0.5) and +(-0.5, -0.5) .. (G-8); 
\draw[] (G-9) .. controls +(0.5, -0.5) and +(-0.5, -0.5) .. (G-10); 
\end{tikzpicture}
\\
\begin{tikzpicture}[scale = 0.35,thick, baseline={(0,-1ex/2)}] 
\tikzstyle{vertex} = [shape = circle, minimum size = 4pt, inner sep=1pt,fill=black] 
\node[vertex] (G--10) at (13.5, -1) [shape = circle, draw] {}; 
\node[vertex] (G-10) at (13.5, 1) [shape = circle, draw] {}; 
\node[vertex] (G--9) at (12.0, -1) [shape = circle, draw] {}; 
\node[vertex] (G--8) at (10.5, -1) [shape = circle, draw] {}; 
\node[vertex] (G--7) at (9.0, -1) [shape = circle, draw] {}; 
\node[vertex] (G-7) at (9.0, 1) [shape = circle, draw] {}; 
\node[vertex] (G--6) at (7.5, -1) [shape = circle, draw] {}; 
\node[vertex] (G-6) at (7.5, 1) [shape = circle, draw] {}; 
\node[vertex] (G--5) at (6.0, -1) [shape = circle, draw] {}; 
\node[vertex] (G-5) at (6.0, 1) [shape = circle, draw] {}; 
\node[vertex] (G--4) at (4.5, -1) [shape = circle, draw] {}; 
\node[vertex] (G-4) at (4.5, 1) [shape = circle, draw] {}; 
\node[vertex] (G--3) at (3.0, -1) [shape = circle, draw] {}; 
\node[vertex] (G-3) at (3.0, 1) [shape = circle, draw] {}; 
\node[vertex] (G--2) at (1.5, -1) [shape = circle, draw] {}; 
\node[vertex] (G-2) at (1.5, 1) [shape = circle, draw] {}; 
\node[vertex] (G--1) at (0.0, -1) [shape = circle, draw] {}; 
\node[vertex] (G-1) at (0.0, 1) [shape = circle, draw] {}; 
\node[vertex] (G-8) at (10.5, 1) [shape = circle, draw] {}; 
\node[vertex] (G-9) at (12.0, 1) [shape = circle, draw] {}; 
\draw[] (G-10) .. controls +(0, -1) and +(0, 1) .. (G--10); 
\draw[] (G--9) .. controls +(-0.5, 0.5) and +(0.5, 0.5) .. (G--8); 
\draw[] (G-7) .. controls +(0, -1) and +(0, 1) .. (G--7); 
\draw[] (G-6) .. controls +(0, -1) and +(0, 1) .. (G--6); 
\draw[] (G-5) .. controls +(0, -1) and +(0, 1) .. (G--5); 
\draw[] (G-4) .. controls +(0, -1) and +(0, 1) .. (G--4); 
\draw[] (G-3) .. controls +(0, -1) and +(0, 1) .. (G--3); 
\draw[] (G-2) .. controls +(0, -1) and +(0, 1) .. (G--2); 
\draw[] (G-1) .. controls +(0, -1) and +(0, 1) .. (G--1); 
\draw[] (G-8) .. controls +(0.5, -0.5) and +(-0.5, -0.5) .. (G-9); 
\end{tikzpicture}
\end{align*}
\caption{Construction of the diagram $d$ from $w(d)$}\label{fig1}
\end{figure}

\begin{thm}[Kauffman \cite{K:90}*{Thm.~4.3}]\label{t:Kauffman}
  Every $n$-diagram $d$ is expressible as a product of the diagrams
  $\hat{e}_i$ ($1 \le i \le n-1$).  If $\Bbbk=\Z[x]$, where $x$ is an
  indeterminate, the map $\sigma$ in equation \eqref{e:morphism} gives
  an isomorphism $\TL_n(x) \cong \D_n(x)$, as $\Z[x]$-algebras.
\end{thm}

\begin{proof}
The first claim follows from Kauffman's algorithm, as illustrated
above (cf.~also \cite{K:90}*{Figure 16} or \cite{PS}*{Thm.~26.10}). It
implies that the map $\sigma$ is surjective. Thus, there is some word
$w$ in $\TL_n(x)$ for which $\sigma(w) = d$.  Furthermore, there is a
reduced word $w'$ such that $w = x^r w'$ for some nonnegative integer
$r$. But $\sigma(w') = x^{r'}d'$, where $d'$ is an $n$-diagram and
$r'$ a nonnegative integer. Hence we have
\[
d = \sigma(w) = x^{r+r'} d'.
\]
Thus $r=r'=0$, $w=w'$, and $d=d'$. By Theorem \ref{t:Jones} we may
assume that $w$ is in normal form.  Finally, any non-trivial linear
relation holding among the normal form elements of $\TL_n(x)$ would
induce a corresponding non-trivial relation among the $n$-diagrams in
$\D_n(x)$, which would be contradictory. This shows that $\sigma$ is
injective, thus an isomorphism.
\end{proof}

Now we return to the general case.

\begin{cor}
Over any commutative ring $\Bbbk$, for any $\delta$ in $\Bbbk$, the
natural map $\sigma$ is an isomorphism $\TL_n(\delta) \cong
\D_n(\delta)$ of $\Bbbk$-algebras. This isomorphism sends the set
$\Omega$ of reduced words in Jones normal form to the set
$n$-diagrams. In particular, $\TL_n(\delta)$ is free as a
$\Bbbk$-module with basis $\Omega$, and $\mathrm{rank}_\Bbbk
\TL_n(\delta) = \Cat_{2n,n} = \frac{1}{n+1}\binom{2n}{n}$.
\end{cor}

\begin{proof}
Regard $\Bbbk$ as a $\Z[x]$-algebra via the specialization map $\Z[x]
\to \Bbbk$ sending $\sum a_i x^i$ to $\sum a_i \delta^i$, for any $a_i \in
\Z$.  Then
\[
\Bbbk \otimes_{\Z[x]} \TL_n(x) \cong \Bbbk \otimes_{\Z[x]} \D_n(x).
\]
In other words, by standard identifications, we have
\[
\TL_n(\delta) \cong \D_n(\delta)
\]
as $\Bbbk$-algebras. This proves the first claim. The other claims
follow.
\end{proof}

Henceforth, we identify $\TL_n(\delta)$ with $\D_n(\delta)$ and $e_i$
with $\hat{e}_i$ for all $i$.

\begin{rmk}
(i) Kauffman works over $\C$ in \cite{K:90}, but his argument works
  the same over $\Z[x]$.

(ii) It follows from the diagrammatic interpretation that
  $\TL_{n-1}(\delta)$ is isomorphic to the subalgebra of
  $\TL_n(\delta)$ generated by $e_1, \dots, e_{n-2}$.

(iii) The diagrammatic interpretation also implies that
  $\TL_n(\delta)$ may be identified with the subalgebra of Brauer's
  centralizer algebra (on $n$ strands, with parameter $\delta$)
  spanned by its planar diagrams. In \cite{Birman-Wenzl}, Birman and
  Wenzl found a presentation of Brauer's algebra that also implies
  this identification.

(iv) See Algorithm~\ref{a:skew} for a very new algorithm that computes
  reduced expressions of $n$-diagrams without any need to apply
  commutation relations.
\end{rmk}

\section{The basic construction}\noindent
The Jones basic construction \cites{Jones:83,Jones:86} was originally
applied to inclusions of von Neumann algebras. It was generalized in
\cite{GHJ}, which we follow here. We work over a field $\Bbbk$ in this
section.  Suppose that $N \subset M$ is a given inclusion of
$\Bbbk$-algebras such that $1_N = 1_M$. Then $M \subset L$, where $L
= \End_N(M)$, is another such inclusion, where $M$ is regarded as a
right $N$-module.  Iterating this idea produces a tower
\begin{equation}
  M_0 \subset M_1 \subset \cdots \subset M_i \subset M_{i+1}
  \subset \cdots
\end{equation}
of $\Bbbk$-algebras, where $M_0=N$, $M_1=M$, and $M_{i+1}
= \End_{M_{i-1}}(M_i)$ for all $i \ge 1$. The \emph{rank}
$\mathrm{rk}(M_k|M_0)$ of $M_k$ over $M_0$ is the smallest number (in
$\N \cup \{\infty\}$) of generators of $M_k$ viewed as a right
$M_0$-module. The \emph{index} $[M:N]$ of $N$ in $M$ is the growth
rate
\[
[M:N] = \limsup_{k\to \infty} \big( \mathrm{rk}(M_k|M_0) \big)^{1/k} . 
\]
If $N \subset M$ is an inclusion of semisimple algebras then
(\cite{GHJ}*{Cor.~2.1.2}) either $[M:N] = 4\cos^2(\pi/q)$ for some
integer $q \ge 3$ or $[M:N] \ge 4$.

Now suppose that $N \subset M$ is an inclusion of finite dimensional
split semisimple $\Bbbk$-algebras. (These are called ``multi-matrix
algebras'' in the terminology of \cite{GHJ}.)  A \emph{trace} on $M$
is a linear map $\tr: M \to \Bbbk$ such that $\tr(xy)=\tr(yx)$ for all
$x,y \in M$. It is \emph{nondegenerate} if the corresponding bilinear
form $(x,y) \mapsto \tr(xy)$ is nondegenerate. Assume that there
exists a nondegenerate trace $\tr$ on $M$ whose restriction to $N$ is
nondegenerate. (This is always true if $\Bbbk$ has characteristic zero.)
Then there is a unique $\Bbbk$-linear map $\E: M \to N$, called a
\emph{conditional expectation}, such that
\[
\begin{alignedat}{2}
  & \E(y)=y && \text{for all } y \in N \\
  & \E(y_1xy_2) = y_1 \E(x) y_2 && \text{for all } x \in M, y_1,y_2 \in N \\
  & \tr(\E(x)) = \tr(x) &\qquad\qquad & \text{for all } x \in M. 
\end{alignedat}
\]
Of course $\E$ is an element of $L = \End_N(M)$.  In this situation,
$L$ is generated by $M$ and $\E$. More precisely, $L$ is generated as a
vector space by all $x_1 \E x_2$, where $x_1,x_2 \in M$.

In general, traces and conditional expectations do not propagate up
the tower.  To obtain such a property it is necessary to consider
Markov traces.  Given $\beta \ne 0$ in $\Bbbk$, a \emph{Markov trace of
modulus} $\beta$ on $N \subset M$ is a nondegenerate trace $\tr$ on
$M$ with nondegenerate restriction to $N$ for which there exists a
(necessarily unique) trace $\Tr$ on $L = \End_N(M)$ such that
\[
\begin{minipage}{1in}
\begin{align*}
  \Tr(x) &= \tr(x) \\
  \beta \Tr(x \E) &= \tr(x) 
\end{align*}
\end{minipage}\qquad\qquad
\text{for all $x \in M$.}
\]
Assuming that a Markov trace $\tr$ of some modulus $0 \ne \beta \in
\Bbbk$ exists on the pair $N \subset M$ of finite dimensional split
semisimple algebras, the authors of \cite{GHJ} show that the trace
propagates up the tower to give Markov traces $\tr_k$ on $M_{k-1}
\subset M_k$ and conditional expectations $\E_k: M_k \to M_{k-1}$, for
each $k\ge 1$. Note that $\E_k$ belongs to $M_{k+1}$. 

\begin{thm}[\cite{GHJ}] \label{t:tower}
Assume that $N \subset M$ is an inclusion of finite dimensional split
semisimple $\Bbbk$-algebras, on which there exists a Markov trace $\tr$
of modulus $\beta$, where $0 \ne \beta \in \Bbbk$. For each $k \ge 1$,
let $\tr_k$, $\E_k$ be as above. Then
\begin{enumerate}
\item $M_k$ is generated by $M_1=M$ and $\E_1, \dots, \E_{k-1}$.
\item The idempotents $\E_1, \dots, \E_{k-1}$ satisfy the relations
  \begin{align*}
    \beta \E_i \E_j \E_i = \E_i \quad &\quad \text{ if } |i-j|=1 \\
    \E_i\E_j=\E_j\E_i &\quad \text{otherwise.}
  \end{align*}
\end{enumerate}
\end{thm}

\begin{rmk}
(i) In \cites{Jones:83,Jones:86} the ground field is $\Bbbk = \C$ and
  the inclusion $N \subset M$ is an inclusion of von Neumann algebras.
  See \cites{Jones:09,EK} for surveys of connections with quantum
  topology and mathematical physics.

(ii) The basic construction was initially applied to construct the Jones
  polynomial \cite{Jones:85} in knot theory. Further applications can
  be found in
  \cites{Birman-Wenzl,Wenzl:annals,Wenzl:93,Kadison,Halverson-Ram}.
  
(iii) In the context of Theorem \ref{t:tower}, the authors of
  \cite{GHJ} show that the pair $N \subset M$ is determined, up to
  isomorphism, by an inclusion matrix $\Lambda$ with nonnegative
  integer entries. The matrix $\Lambda$ may be encoded as a graph, the
  Bratteli diagram of the pair, and $[M:N]$ is the square of the
  Euclidean norm of the graph. It follows that $[M:N] \le 4$ if and
  only if the Bratteli diagram of the pair $N \subset M$ is a Coxeter
  graph of type A, D, or E.
\end{rmk}

Theorem \ref{t:tower} focuses attention on the $\Bbbk$-algebra
$\A_n(\beta)$ defined by the generators $1, u_1, \dots, u_{n-1}$
subject to the relations
\begin{equation}\label{e:Jones}
  \begin{aligned}
  u_i^2 = u_i, \quad
  \beta u_i u_j u_i = u_i \text{ if } |i-j|=1, \quad
  u_i u_j &= u_j u_i \text{ if } |i-j|>1 .
\end{aligned}
\end{equation}
We call this algebra the \emph{Jones algebra}.
We now consider the issue of semisimplicity of the Jones algebra,
following \cite{GHJ}.  Let $x$ be an indeterminate and $\Z[x]$ the
ring of polynomials in $x$ with integer coefficients. Define
polynomials $P_n(x)$ in $\Z[x]$ for each integer $n \ge 0$ by the
recursion
\begin{equation}
\begin{gathered}
  P_0(x)=1, \qquad P_1(x)=1,\\ P_{n+1}(x) = P_n(x)-xP_{n-1}(x)
  \quad \text{if $n \ge 1$}.
\end{gathered}
\end{equation}
The analysis in \cite{GHJ} suggests that the $P_n(x)$ are closely
related to the Chebyshev polynomials of the second kind. The precise
relation was made explicit in \cite{Benkart-Moon}.

Say that a trace $\tr_n:\A_n(\beta) \to \Bbbk$ is \emph{normalized} if
it takes the identity to the identity in $\Bbbk$.  Here is the
semisimplicity result. It was originally obtained over $\C$ by Jones,
and extended to arbitrary fields in the cited reference.

\begin{thm}[\cite{GHJ}*{Prop.~2.8.5}]\label{t:ss}
  Suppose that $\Bbbk$ is a field, $0 \ne \beta \in \Bbbk$, and $P_1(\beta^{-1})
  P_2(\beta^{-1}) \cdots P_{n-1}(\beta^{-1}) \ne 0$ in $\Bbbk$.  Then:
  \begin{enumerate}
  \item The Jones algebra $\A_n(\beta)$ is split semisimple over
    $\Bbbk$.
  \item There exists a unique normalized trace $\tr_n:\A_n(\beta) \to
    \Bbbk$ such that for any $1 \le j \le n-1$,
    \[
    \beta \tr_n(wu_j) = \tr_n(w)
    \]
    for all $w$ in the subalgebra generated by $1, u_1, \dots,
    u_{j-1}$. Furthermore, $\tr_n$ is nondegenerate if
    $P_n(\beta^{-1}) \ne 0$.
  \item The natural map $\A_{n-1}(\beta) \to \A_n(\beta)$ is injective
    and $\tr_n$ extends $tr_{n-1}$.
  \end{enumerate}
\end{thm}

The detailed proof of Theorem~\ref{t:ss} in \cite{GHJ} reveals
that the associated conditional expectation $\E_n: \A_n(\beta) \to
\A_{n-1}(\beta)$ satisfies the identities
\begin{equation}\label{e:CE-on-gens}
  \E_n(u_{n-1}) = \beta^{-1} 1\quad\text{and}\quad
  \E_n(u_j) = u_j \text{ for all } 1\le j < n-1.
\end{equation}
The first equality follows from $\E_n u_{n-1} \E_n = \beta^{-1} \E_n$ and
the latter is by definition of conditional expectation. Furthermore,
$\E_n$ satisfies
\begin{equation}\label{e:CE}
  u_{n-1} \E_n(u_{n-1} x) = \beta^{-1} u_{n-1} x, \quad \text{ for all
  } x \in \A_n(\beta).
\end{equation}
Finally, by setting $w=1$ in part (b) of Theorem~\ref{t:ss} we obtain
\begin{equation}\label{e:trace}
  \tr_n(u_j) = \beta^{-1} \text{ for all $1\le j \le n-1$}.
\end{equation}

\begin{rmk}
It may be useful to keep in mind that the conditional expectation
$\E_n$, viewed as an endomorphism of $\A_n(\beta)$, is identified
with the idempotent $u_n$ inside the next algebra $\A_{n+1}$ in the
tower construction.
\end{rmk}

\section{Relation between $\A_n$ and $\TL_n$}\noindent
We continue to assume that $\Bbbk$ is a field in this section.
If $\beta \ne 0$, by setting $e_i = \delta u_i$ for all $i$ we see
that relations \eqref{e:TL} and \eqref{e:Jones} are formally
equivalent if and only if $\beta = \delta^2$. Thus, we have the
following.

\begin{prop}\label{p:TL-iso}
  Suppose that $\Bbbk$ is a field. For any $\delta \ne 0$ in
  $\Bbbk$, there is a $\Bbbk$-algebra isomorphism $\A_n(\delta^2) \cong
  \TL_n(\delta)$ given by $u_i \mapsto \delta^{-1} e_i$ for all $i$.
\end{prop}

In other words, as long as $\beta \ne 0$, the Jones algebra is an
equivalent form of the Temperley--Lieb algebra in which the generators
have been rescaled to idempotents.

\begin{rmk}
On the other hand, it follows from its defining presentation
\eqref{e:Jones} that for all $n>2$, $\A_n(\beta)$ collapses to the zero
algebra at $\beta=0$, while the dimension of $\TL_n(0)$ is the $n$th
Catalan number. Also, $\A_2(0) \ncong \TL_2(0)$ is clear. So the
Jones algebra gives no information about $\TL_n(0)$.
\end{rmk}

We now consider the implications of Proposition~\ref{p:TL-iso} for
traces and conditional expectations under Kauffman's diagrammatic
interpretation of the Temperley--Lieb algebra. Always assuming that $0
\ne \beta = \delta^2$, we see that the trace $\tr_n$ and conditional
expectation $\E_n$ considered at the end of the previous section carry
over under the isomorphism $\A_n(\delta^2) \cong \TL_n(\delta)$ to give
a trace and conditional expectation on $\TL_n(\delta)$, that we will
denote by the same symbols.  It is convenient to renormalize so that
the trace still takes identity to identity. With that renormalization,
it turns out that for any $n$-diagram $d$,
\begin{equation}\label{e:diag-tr}
  \tr_n(d) = \delta^{-n}\bar{d} \quad\text{and}\quad
  \E_n(d) = \delta^{-1}\bar{d}^{(n)}
\end{equation}
where $\bar{d}$ is the diagram obtained from $d$ by drawing
non-intersecting curves outside the enclosing rectangle from vertex
$i$ to vertex $i'$ for all $i = 1,\dots, n$, and $\bar{d}^{(n)}$ is
the diagram obtained from $d$ by drawing a single such curve from
vertex $n$ to vertex $n'$.  (See \cite{Kauffman-Lins}*{p.~10}.) In
this process, loops are replaced by $\delta$. We can visualize
$\bar{d}$ and $\bar{d}^{(n)}$ by the pictures:
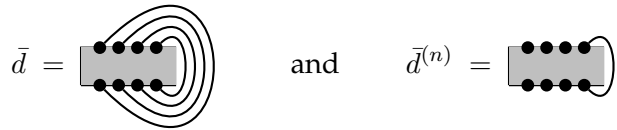
\[
\bar{d} \;=\; 
\begin{tikzpicture}[scale = 0.25,thick, baseline={(0,-1ex/2)}]
  \tikzstyle{vertex} = [shape=circle,minimum size=4pt,inner sep=1pt,fill=black]
  \path[use as bounding box] (0, -1) rectangle (5, 1);
  \draw[thick] (0,-1)--(5,-1)--(5,1)--(0,1)--(0,-1);
  \fill[lightgray] (0,-1) rectangle (5,1);
  \node[vertex] (B-1) at (1, -1) [shape = circle, draw] {};
  \node[vertex] (B-2) at (2, -1) [shape = circle, draw] {};
  \node[vertex] (B-3) at (3, -1) [shape = circle, draw] {};
  \node[vertex] (B-4) at (4, -1) [shape = circle, draw] {};
  \node[vertex] (T-1) at (1, 1) [shape = circle, draw] {};
  \node[vertex] (T-2) at (2, 1) [shape = circle, draw] {};
  \node[vertex] (T-3) at (3, 1) [shape = circle, draw] {};
  \node[vertex] (T-4) at (4, 1) [shape = circle, draw] {};
  \draw[] (B-4) .. controls +(2, -2) and +(2,2) .. (T-4);
  \draw[] (B-3) .. controls +(4, -4) and +(4,4) .. (T-3);
  \draw[] (B-2) .. controls +(6, -6) and +(6,6) .. (T-2);
  \draw[] (B-1) .. controls +(8, -8) and +(8,8) .. (T-1);
\end{tikzpicture}
\qquad\qquad\text{and}\qquad\; \bar{d}^{(n)} \;=\;
\begin{tikzpicture}[scale = 0.25,thick, baseline={(0,-1ex/2)}]
  \tikzstyle{vertex} = [shape = circle, minimum size = 4pt,
    inner sep=1pt,fill=black] 
  \draw[thick] (0,-1)--(5,-1)--(5,1)--(0,1)--(0,-1);
  \fill[lightgray] (0,-1) rectangle (5,1);
  \node[vertex] (B-1) at (1, -1) [shape = circle, draw] {};
  \node[vertex] (B-2) at (2, -1) [shape = circle, draw] {};
  \node[vertex] (B-3) at (3, -1) [shape = circle, draw] {};
  \node[vertex] (B-4) at (4, -1) [shape = circle, draw] {};
  \node[vertex] (T-1) at (1, 1) [shape = circle, draw] {};
  \node[vertex] (T-2) at (2, 1) [shape = circle, draw] {};
  \node[vertex] (T-3) at (3, 1) [shape = circle, draw] {};
  \node[vertex] (T-4) at (4, 1) [shape = circle, draw] {};
  \draw[] (B-4) .. controls +(2, -2) and +(2,2) .. (T-4);
\end{tikzpicture}
\]
respectively.
It follows from equation \eqref{e:diag-tr} that on generators the maps
$\E_n$ and $\tr_n$ satisfy the identities in equations
\eqref{e:CE-on-gens} and \eqref{e:trace}, respectively, with $\beta$
replaced by $\delta$ and $u_i$ replaced by $e_i$ for all~$i$.

Now we consider implications for the semisimplicity of
$\TL_n(\delta)$. Our goal is to recast the hypothesis of
Theorem~\ref{t:ss} in a more palatable form.  Let $q \in \Bbbk$. For a
positive integer $n$, the classical Gaussian integer $\dbracket{n}_q$
is
\[
  \dbracket{n}_q = 1+q+q^2 + \cdots +q^{n-1}.
\]  
If $q \ne 1$, it can be written in the form $\dbracket{n}_q =
(1-q^n)/(1-q)$ but the definition of $\dbracket{n}_q$ makes perfect
sense at $q = 1$, where it evaluates to the integer $n$. It is
customary to set $\dbracket{0}_q = 0$. Let
\[
\dbracket{n}_q^! = \dbracket{1}_q \cdots \dbracket{n-1}_q
\dbracket{n}_q = \textstyle \prod_{k=1}^n \dbracket{k}_q
\]
if $n>0$, and set $\dbracket{0}_q^! = 1$. 

Now choose $q$ in $\Bbbk$ such that $q \ne 0$, $q \ne -1$, and $\beta
= q+2+q^{-1}$. (Replace $\Bbbk$ by a suitable quadratic extension if
necessary.) It follows by a simple induction that
\begin{equation}
P_n(\beta^{-1}) = \frac{1+q+q^2 + \cdots + q^n}{(1+q)^n} =
\frac{\dbracket{n+1}_q}{(1+q)^n}.
\end{equation}
This was observed in Prop.~2.8.3(iv) of \cite{GHJ}.  Then
Theorem~\ref{t:ss} gives the following corollary.

\begin{cor}\label{c:ss}
  Suppose that $q \ne 0$, $q \ne -1$ where $q$ is in the field
  $\Bbbk$. With $\beta = q+2+q^{-1}$, the Jones algebra $\A_n(\beta) =
  \A_n(q+2+q^{-1})$ is split semisimple over $\Bbbk$ whenever
  $\dbracket{n}^!_q \ne 0$.
\end{cor}

If $\beta = q + 2 + q^{-1}$ then $\beta^{1/2} = \pm(q^{1/2} +
q^{-1/2})$, provided that a square root of $q$ exists in $\Bbbk$.
This, in light of Proposition~\ref{p:TL-iso}, gives the following
restatement of Corollary~\ref{c:ss}.

\begin{cor}
  Let $\Bbbk$ be a field containing a square root $q^{1/2}$ of $q$,
  where $q \ne 0$, $q \ne -1$. If $\dbracket{n}^!_q \ne 0$ then
  $\TL_n(\pm(q^{1/2}+q^{-1/2}))$ is split semisimple over $\Bbbk$.
\end{cor}

The advent of the theory of quantum groups led to a slightly different
normalization of the classical Gaussian integers, as follows.
First, we set $v = q^{1/2}$ so that $q = v^2$. We will always assume that
$q=v^2$ from now on. Under that assumption, we have
\[
  \dbracket{n}_q = \dbracket{n}_{v^2} = 1 + v^2 + v^4 + \cdots +
  v^{2(n-1)} = v^{n-1} \textstyle \sum_{k=0}^{n-1} v^{-(n-1)+2k}.
\]
For any $n>0$, the \emph{balanced} form $[n]_v$ of the Gaussian integer,
which is also known as the \emph{quantum integer} (or $q$-integer)
corresponding to $n$, is defined by
\[
  [n]_v = v^{-(n-1)} + v^{-(n-1)+2} + \cdots  + v^{n-1}
    = \textstyle \sum_{k=0}^{n-1} v^{-(n-1)+2k}. 
\]
The definition of $[n]_v$ makes sense when $v=1$, in which case it
evaluates to $n$. (We also set $[0]_v = 0$.)  Notice that if $v^2
\ne 1$ then $[n]_v = \frac{v^n - v^{-n}}{v-v^{-1}}$.  We
define $[n]_v^!  = [1]_v \cdots[n-1]_v [n]_v$ and set $[0]^!_v = 1$.
The balanced and classical forms of Gaussian integers are related by
\begin{equation}
  \dbracket{n}_q = v^{n-1} [n]_v  \qquad (\text{for } q=v^2).
\end{equation}
As $[n]_v$ and $\dbracket{n}_q$ are the same up to a power of $v$, the
preceding corollary may be restated in the following form.

\begin{cor}\label{c:ss-crit}
  Let $\Bbbk$ be a field, $0 \ne v \in \Bbbk$, where $0 \ne
  v+v^{-1}$. If $[n]_v^! \ne 0$ then $\TL_n(\pm(v+v^{-1}))$ is split
  semisimple over $\Bbbk$.
\end{cor}

See \cite{DG:orthog} for a new elementary proof of this result. The
recent paper \cite{AST} gives a very different proof based on tilting
modules. In many of the early references, e.g.,
\cites{GHJ,Martin,Westbury}, semisimplicity criteria were
formulated in a more complicated way than the simple condition in
Corollary~\ref{c:ss-crit}.

\section{$\TL_n$ as a quotient of the Iwahori--Hecke algebra}%
\label{s:Hecke}\noindent
Over the complex field $\C$, the observation that $\TL_n(v+v^{-1})$ is
isomorphic to a quotient of the Iwahori--Hecke algebra goes back (at
least) to Jones \cite{Jones}.  The following result is a slight
extension (with a different normalization) of
\cite{GHJ}*{Prop.~2.11.1}.

\begin{prop}\label{p:TL-alt}
Let $\Bbbk$ be a commutative unital ring with $v \in \Bbbk$ a fixed
invertible element. Set $\gamma_i = e_i - v^{-1}$ for all $i =
1,\dots, n-1$.  The Temperley--Lieb algebra $\TL_n(\delta)$, with
parameter $\delta = v+v^{-1}$, is the algebra defined by the
generators $\gamma_1, \dots, \gamma_{n-1}$ subject to the relations
\begin{enumerate}
\item $(\gamma_i+v^{-1})(\gamma_i-v)=0$.
\item $\gamma_i\gamma_{i+1}\gamma_i = \gamma_{i+1}\gamma_i\gamma_{i+1}$.
\item $\gamma_i \gamma_j = \gamma_j\gamma_i$ if $|i-j|>1$.
\item $v^3\gamma_i\gamma_{i+1}\gamma_i + v^2(\gamma_i\gamma_{i+1} +
  \gamma_{i+1}\gamma_i) + v(\gamma_i+\gamma_{i+1}) + 1 = 0$. 
\end{enumerate}
\end{prop}

\begin{proof}
One easily checks that the relation $e_i^2=\delta e_i$ is equivalent
to (a). Thus, we may replace $\gamma_i^2$ by $(v-v^{-1})\gamma_i + 1$
in the expansion
\begin{align*}
e_i&e_{i+1}e_i  \\ &= \gamma_i\gamma_{i+1}\gamma_i +
v^{-1}(\gamma_i\gamma_{i+1} + \gamma_{i+1}\gamma_i + \gamma_i^2) +
v^{-2}(2\gamma_i + \gamma_{i+1}) + v^{-3}
\intertext{to obtain the simplification}
&= \gamma_i\gamma_{i+1}\gamma_i +
v^{-1}(\gamma_i\gamma_{i+1} + \gamma_{i+1}\gamma_i) +
v^{-2}(\gamma_i + \gamma_{i+1}) + v^{-3} + \gamma_i + v^{-1}.
\end{align*}
It follows that the relation $e_ie_{i+1}e_i - e_i = 0$ is equivalent to
\[
\gamma_i\gamma_{i+1}\gamma_i + v^{-1}(\gamma_i\gamma_{i+1} +
\gamma_{i+1}\gamma_i) + v^{-2}(\gamma_i+\gamma_{i+1}) + v^{-3} = 0.
\]
This in turn is equivalent to (d). Interchanging $i$ and $i+1$ in the
argument shows that the relation $e_{i+1}e_ie_{i+1} - e_{i+1}= 0$ is
equivalent to
\[
\gamma_{i+1}\gamma_i\gamma_{i+1} + v^{-1}(\gamma_i\gamma_{i+1} +
\gamma_{i+1}\gamma_i) + v^{-2}(\gamma_i+\gamma_{i+1}) + v^{-3} = 0.
\]
Comparing the last two equivalences shows that (b) holds. Finally, (c)
is clear. On the other hand, if one starts with elements $\gamma_i$
satisfying relations (a)--(d) then by setting $e_i = \gamma_i+v^{-1}$
the defining relations \eqref{e:TL} for $\TL_n(\delta)$ may be
deduced.
\end{proof}

We continue to work over a commutative ring $\Bbbk$ with $1$.
Recall \cite{Jimbo} (see also \cite{Lu:Hecke}) that the Iwahori--Hecke
algebra $\HH_n$ of type A may be defined as the $\Bbbk$-algebra with
$1$ on generators $T_1, \dots, T_{n-1}$ subject to the relations
\begin{equation}\label{e:Hecke}
\begin{gathered}
  (T_i+v^{-1})(T_i-v) = 0\\
  T_iT_jT_i = T_jT_iT_j \text{ if } |i-j|=1\\
  T_iT_j = T_jT_i \text{ if } |i-j|>1. 
\end{gathered}
\end{equation}
The generators $T_i$ are invertible, with
\begin{equation*}
T_i^{-1} = T_i+v^{-1}-v.
\end{equation*}
We immediately have the following consequence of
Proposition~\ref{p:TL-alt}; compare with \cite{GHJ}*{Cor.~2.11.2}.

\begin{cor}\label{c:TL-quo}
Suppose that $\Bbbk$ is a unital commutative ring containing an
invertible element $v$, and that $\delta = v+v^{-1}$.  There exists a
surjective algebra homomorphism
\[
\psi_n: \HH_n \to \TL_n(\delta)
\]
defined by $\psi_n(T_i)=\gamma_i=e_i-v^{-1}$ for $i = 1, \dots, n-1$.
If $n=1$ or $n=2$ it is an isomorphism.  If $n \ge 3$, the kernel of
$\psi_n$ is the two-sided ideal of $\HH_n$ generated by
\[
v^3 T_1T_2T_1 + v^2(T_1T_2+T_2T_1) + v(T_1+T_2) + 1. 
\]
Furthermore, the diagram (in which the horizontal maps are the canonical
inclusions)
\[
\begin{tikzcd}
\HH_n \arrow[r] \arrow[d, "\psi_n"] & \HH_{n+1} \arrow[d, "\psi_{n+1}"] \\
\TL_n(\delta) \arrow[r] & \TL_{n+1}(\delta)
\end{tikzcd}
\]
is commutative.
\end{cor}

\begin{proof}
The existence of $\psi_n$ follows from the definition of $\HH_n$ and
Proposition~\ref{p:TL-alt}. It is easy to check that $\psi_1$ and
$\psi_2$ are isomorphisms, and that the kernel of $\psi_n$ for $n \ge
3$ is generated by all
\[
x_i = v^3 T_iT_{i+1}T_i + v^2(T_iT_{i+1}+T_{i+1}T_i) + v(T_i+T_{i+1}) + 1
\]
for $i = 1, \dots, n-2$. Each $T_i$ is invertible, with $T_i^{-1} =
T_i+v^{-1}-v$. From the braid relations for the $T_i$ we have
\[
(T_1T_2 \cdots T_{n-1}) T_k (T_{n-1}^{-1} \cdots T_2^{-1}T_1^{-1}) = T_{k+1}
\]
for all $k = 1, \dots, n-2$. (This is no typo; it really is necessary to
conjugate by the fixed element $T_1 \cdots T_{n-1}$ for each $k$.)
Thus,
\[
(T_1T_2 \cdots T_{n-1}) x_k (T_{n-1}^{-1} \cdots T_2^{-1}T_1^{-1}) = x_{k+1}
\]
for all $k = 1, \dots, n-3$. This shows that the kernel is generated
by $x_1$, as required. The final claim is obvious.
\end{proof}

\begin{rmk}\label{r:HH2}
The original version of $\HH_n$ in the literature
(see \cites{KL:79,DJ1,DJ2,DJ3,DJ4,Murphy1,Murphy2})
was defined with the quadratic relation 
\[
(T_i+1)(T_i-q) = 0 \quad\text{where}\quad q=v^2
\]
and with the remaining relations the same. This leads to an isomorphic
algebra, but many formulas look different. To enable comparisons
between different versions, it is convenient to define (see \cites{BW,
  Bigelow}) the two-parameter Iwahori--Hecke algebra $\HH_n(q_1,q_2)$
to be the $\Bbbk$-algebra with $1$ defined by generators $T_1, \dots,
T_{n-1}$ with the defining relations
\begin{equation*}
  \begin{gathered}
    (T_i-q_1)(T_i-q_2) = 0 \\
  T_iT_jT_i = T_jT_iT_j \text{ if } |i-j|=1 \\
  T_iT_j = T_jT_i \text{ if } |i-j|>1 .
\end{gathered}
\end{equation*}
In this notation, the original version of $\HH_n$ is $\HH_n(-1,q)$ and
the version defined in \eqref{e:Hecke} is $H_n(-v^{-1},v)$, where
$q=v^2$. The algebra map defined by $T_i \mapsto v^{-1}T_i$ defines an
isomorphism
\[
\HH_n(-v^{-1},v) \cong \HH_n(-1,v^2).
\]
Assuming that $q_1$ and $q_2$ are invertible in $\Bbbk$ and setting $q
= -q_2/q_1$, one has an algebra isomorphism
\[
\HH_n(-1,q) \cong \HH_n(q_1,q_2)
\]
defined by $T_i \mapsto -q_1^{-1}T_i$ for all $i$. This means that
$\TL_n(v+v^{-1})$ can be constructed as a quotient of any version of
$\HH_n$, provided only that the eigenvalues of the $T_i$ are
invertible in $\Bbbk$ (and $q=v^2$).
\end{rmk}

\section{Skew shapes and $321$-avoiding permutations}\label{s:skew}%
\noindent
The purpose of this section is describe a very new algorithm that
efficiently computes the Jones normal form in Theorem \ref{t:Jones}
(and the dual version in Remark~\ref{r:dual-JNF}) corresponding to a
given $n$-diagram, without the need to apply commutation
relations. The algorithm is due to Chris Bowman and first appeared in
\cite{Bowman-et-al} in a more general context; see also
\cite{Bowman:book}*{Thm.~5.2.3}. The version given here has been mildly
adapted. As usual, we identify partitions $\lambda$ with Young
diagrams (shapes). Recall \cites{Macd,Fulton} that a skew shape
$\lambda \setminus \mu$ for partitions $\lambda$, $\mu$ with $\mu
\subset \lambda$ is defined as their set-theoretic difference.

Fix an origin $(0,0)$ in the euclidean plane $\R^2 = \R \times \R$.
Establish compass directions for the plane in which north points to
the top of the page.  Pick an orthonormal coordinate system for $\R^2$
in which the positive directions are \emph{south} and \emph{east},
respectively, so that for $x>0$, $y>0$ the point $(x,y)$ is located
$x$ units south and $y$ units east of the origin. These choices are
dictated by the usual ``English notation'' for tableaux, which are
regarded as consisting of rows and columns numbered similarly to the
way matrix entries are numbered.

\begin{algo}\label{a:skew}
Let $d$ be a given $n$-diagram. Working from left to right, number its
northern vertices by $1,\dots, n$ and number its southern vertices by
$1', \dots, n'$. For any $k$ in $\{1, \dots, n\}$, define
\[
\xi(k) = 
\begin{cases}
  W & \text{if the vertex $k$ is connected to a
    vertex strictly to its right}\\
  S & \text{otherwise}
\end{cases}
\]
and
\[
\xi(k') = 
\begin{cases}
  N & \text{if the vertex $k'$ is connected to a
    vertex weakly to its right}\\
  E & \text{otherwise.}
\end{cases}
\]
Starting at the vertex $1$ in the upper left corner of $d$ and proceeding
clockwise through the vertices, we obtain the vector
\[
\xi(d) = (\xi(1), \dots, \xi(n), \xi(n'), \dots, \xi(1'))
\]
that records the sequence of compass directions. 
The vector $\xi(d)$ determines a closed polygonal path in $\R \times
\R$ (having vertices in $\Z \times \Z$) as follows:
\begin{itemize}
\item Start at the point $(0,0)$.
\item Following the compass directions in the sequence $\xi(d)$, move
  one unit in the prescribed direction at each step.
\end{itemize}
The path always consists of $2n$ unit length segments.  By discarding
all unit length segments which are traversed twice, we obtain a unique
skew shape $\lambda \setminus \mu$. Fill each unit box in the skew
shape with the number $\omega(i,j) = i-j$ where $(i,j)$ is the
southeastern point of the box.  With this filling, we obtain a labeled
skew shape $(\lambda \setminus \mu, \omega)$ in the sense of
\cite{BJS93}*{p.~363}.
\end{algo}

\begin{example}\label{ex:polygonal}
The $9$-diagram $d$ displayed below
\[
d \; = \;
\begin{tikzpicture}[scale = 0.35,thick, baseline={(0,-1ex/2)}] 
  \tikzstyle{vertex} = [shape = circle, minimum size = 4pt,
    inner sep = 1pt, fill=black] 
\node[vertex] (G--9) at (12.0, -1) [shape = circle, draw] {}; 
\node[vertex] (G--8) at (10.5, -1) [shape = circle, draw] {}; 
\node[vertex] (G--7) at (9.0, -1) [shape = circle, draw] {}; 
\node[vertex] (G-7) at (9.0, 1) [shape = circle, draw] {}; 
\node[vertex] (G--6) at (7.5, -1) [shape = circle, draw] {}; 
\node[vertex] (G-6) at (7.5, 1) [shape = circle, draw] {}; 
\node[vertex] (G--5) at (6.0, -1) [shape = circle, draw] {}; 
\node[vertex] (G-1) at (0.0, 1) [shape = circle, draw] {}; 
\node[vertex] (G--4) at (4.5, -1) [shape = circle, draw] {}; 
\node[vertex] (G--1) at (0.0, -1) [shape = circle, draw] {}; 
\node[vertex] (G--3) at (3.0, -1) [shape = circle, draw] {}; 
\node[vertex] (G--2) at (1.5, -1) [shape = circle, draw] {}; 
\node[vertex] (G-2) at (1.5, 1) [shape = circle, draw] {}; 
\node[vertex] (G-3) at (3.0, 1) [shape = circle, draw] {}; 
\node[vertex] (G-4) at (4.5, 1) [shape = circle, draw] {}; 
\node[vertex] (G-5) at (6.0, 1) [shape = circle, draw] {}; 
\node[vertex] (G-8) at (10.5, 1) [shape = circle, draw] {}; 
\node[vertex] (G-9) at (12.0, 1) [shape = circle, draw] {}; 
\draw[] (G--9) .. controls +(-0.5, 0.5) and +(0.5, 0.5) .. (G--8); 
\draw[] (G-7) .. controls +(0, -1) and +(0, 1) .. (G--7); 
\draw[] (G-6) .. controls +(0, -1) and +(0, 1) .. (G--6); 
\draw[] (G-1) .. controls +(1, -1) and +(-1, 1) .. (G--5); 
\draw[] (G--4) .. controls +(-0.8, 0.8) and +(0.8, 0.8) .. (G--1); 
\draw[] (G--3) .. controls +(-0.5, 0.5) and +(0.5, 0.5) .. (G--2); 
\draw[] (G-2) .. controls +(0.5, -0.5) and +(-0.5, -0.5) .. (G-3); 
\draw[] (G-4) .. controls +(0.5, -0.5) and +(-0.5, -0.5) .. (G-5); 
\draw[] (G-8) .. controls +(0.5, -0.5) and +(-0.5, -0.5) .. (G-9); 
\end{tikzpicture}
\]
is associated by Algorithm \ref{a:skew} to the sequence
\[
\xi(d) = (W,W,S,W,S,S,S,W,S,E,N,N,N,E,E,E,N,N).
\]
The sequence corresponds to the polygonal path traced out in the
figure on the left below (the origin is in its upper right corner)
\[
\begin{minipage}{1.6cm}
\begin{tikzpicture}[scale = 0.4,thick, baseline={(0,-1ex/2)}] 
  \draw[very thick]
(0,0)--(-1,0)--(-2,0)--(-2,-1)--(-3,-1)--(-3,-2)--(-3,-3)--(-3,-4)--(-4,-4);
  \draw[very thick]
(-4,-4)--(-4,-5)--(-3,-5)--(-3,-4)--(-3,-3)--(-3,-2)--(-2,-2)--(-1,-2)--(0,-2);
  \draw[very thick]
(0,-2)--(0,-1)--(0,0);
  \draw[step=1,dashed,gray,very thin] (-4,-5) grid (0,0);
\end{tikzpicture}
\end{minipage}
\qquad \qquad
\begin{minipage}{1.6cm}
\ytableausetup{boxsize=0.4cm}
\begin{ytableau}
  *(lightgray) & *(lightgray) & 2 & 1\\
  *(lightgray) & 4 & 3 & 2\\
  *(lightgray)\\
  *(lightgray)\\
  8
\end{ytableau}
\end{minipage}
\]
and its corresponding labeled skew shape $(\lambda \setminus \mu,
\omega)$ is displayed in the figure on the right above. Here we have
taken $\lambda = (4^2,1^3)$ and $\mu = (2,1^3)$ in the standard
exponential notation for partitions. (Whenever we display skew shapes
$\lambda \setminus \mu$, we will shade the boxes in $\mu$, as we did
above.)
\end{example}

\begin{rmk}
The alert reader will have noticed that the segments that get
discarded (the ones which are traversed twice) in
Algorithm~\ref{a:skew} always correspond to vertical edges (edges
pairing $i$ and $i'$) in the $n$-diagram $d$, which are labeled by $S$
and $N$ respectively in the sequence $\xi(d)$. The algorithm can be
reformulated by using $W$ and $E$ labels instead; this leads to an
equivalent theory. In fact, one can randomly label some of those
twice-traversed vertical edges by $S$ and $N$ and the rest by $W$ and
$E$ without sacrificing anything. For instance, the diagram $d$ in
Example~\ref{ex:polygonal} has the following associated polygonal
paths
\[
\begin{tikzpicture}[scale = 0.4,thick, baseline={(0,-1ex/2)}] 
  \draw[very thick]
(0,0)--(-1,0)--(-2,0)--(-2,-1)--(-3,-1)--(-3,-2)--(-4,-2)--(-5,-2)--(-6,-2);
  \draw[very thick]
(-6,-2)--(-6,-3)--(-5,-3)--(-5,-2)--(-4,-2)--(-3,-2)--(-2,-2)--(-1,-2)--(0,-2);
  \draw[very thick]
(0,-2)--(0,-1)--(0,0);
  \draw[step=1,dashed,gray,very thin] (-6,-3) grid (0,0);
\end{tikzpicture}
\qquad
\begin{tikzpicture}[scale = 0.4,thick, baseline={(0,-1ex/2)}] 
  \draw[very thick]
(0,0)--(-1,0)--(-2,0)--(-2,-1)--(-3,-1)--(-3,-2)--(-4,-2)--(-4,-3)--(-5,-3);
  \draw[very thick]
(-5,-3)--(-5,-4)--(-4,-4)--(-4,-3)--(-4,-2)--(-3,-2)--(-2,-2)--(-1,-2)--(0,-2);
  \draw[very thick]
(0,-2)--(0,-1)--(0,0);
  \draw[step=1,dashed,gray,very thin] (-5,-4) grid (0,0);
\end{tikzpicture}
\qquad
\begin{tikzpicture}[scale = 0.4,thick, baseline={(0,-1ex/2)}] 
  \draw[very thick]
(0,0)--(-1,0)--(-2,0)--(-2,-1)--(-3,-1)--(-3,-2)--(-3,-3)--(-4,-3)--(-5,-3);
  \draw[very thick]
(-5,-3)--(-5,-4)--(-4,-4)--(-4,-3)--(-3,-3)--(-3,-2)--(-2,-2)--(-1,-2)--(0,-2);
  \draw[very thick]
(0,-2)--(0,-1)--(0,0);
  \draw[step=1,dashed,gray,very thin] (-5,-4) grid (0,0);
\end{tikzpicture}
\qquad
\begin{tikzpicture}[scale = 0.4,thick, baseline={(0,-1ex/2)}] 
  \draw[very thick]
(0,0)--(-1,0)--(-2,0)--(-2,-1)--(-3,-1)--(-3,-2)--(-3,-3)--(-3,-4)--(-4,-4);
  \draw[very thick]
(-4,-4)--(-4,-5)--(-3,-5)--(-3,-4)--(-3,-3)--(-3,-2)--(-2,-2)--(-1,-2)--(0,-2);
  \draw[very thick]
(0,-2)--(0,-1)--(0,0);
  \draw[step=1,dashed,gray,very thin] (-4,-5) grid (0,0);
\end{tikzpicture}
\]
depending on the four possible choices of labeling of its two
twice-traversed vertical edges; these choices give different labeled skew
shapes
\[
\ytableausetup{boxsize=1em}
\begin{ytableau}
  *(lightgray) & *(lightgray) & *(lightgray) & *(lightgray) & 2 & 1\\
  *(lightgray) & *(lightgray) & *(lightgray) & 4 & 3 & 2\\
  8
\end{ytableau}\qquad
\begin{ytableau}
  *(lightgray) & *(lightgray) & *(lightgray) & 2 & 1\\
  *(lightgray) & *(lightgray) & 4 & 3 & 2\\
  *(lightgray)\\
  8
\end{ytableau}\qquad
\begin{ytableau}
  *(lightgray) & *(lightgray) & *(lightgray) & 2 & 1\\
  *(lightgray) & *(lightgray) & 4 & 3 & 2\\
  *(lightgray) & *(lightgray)\\
  8
\end{ytableau}\qquad
\begin{ytableau}
  *(lightgray) & *(lightgray) & 2 & 1\\
  *(lightgray) & 4 & 3 & 2\\
  *(lightgray)\\
  *(lightgray)\\
  8
\end{ytableau}
\]
as shown above. Such ambiguities are addressed by Definition
\ref{d:BJS} below (and explain the need for including it). Notice that
all four of the labeled skew shapes have the same row reading sequence
$(2,1,4,3,2,8)$ and the same column reading sequence
$(8,4,2,3,1,2)$. Here, by \emph{row reading sequence} we mean the
sequence of labels read in order across the rows from first to last;
similarly, the \emph{column reading sequence} is defined by the labels
read in order down the columns taken from left to right.
\end{rmk}

Given a labeled skew shape, its \emph{row (resp., column)}
\emph{reading word} is the product $e_{i_1} e_{i_2} \dots e_{i_l}$
corresponding to its row (resp., column) reading sequence $(i_1, i_2,
\dots, i_l)$.

We will identify a given skew shape $\theta = \lambda \setminus \mu$
with a subset of $\Z \times \Z$ by embedding $\theta$ in $\R^2$ as
above, so that corner points of boxes lie in $\Z \times \Z$, and by
identifying each box in $\theta$ with the coordinate pair of its
southeastern (i.e., lower right) corner point. We need the following
definition, based on \cite{BJS93}*{ p.~363}.

\begin{defn}\label{d:BJS}
We say that labeled skew shapes $(\lambda \setminus \mu, \omega)$ and
$(\alpha \setminus \beta, \omega)$ are \emph{BJS-equivalent} if there
exists an order-preserving (we can use the product order on $\Z
\times \Z$) bijection
\[
f: \lambda \setminus \mu \to \alpha \setminus \beta
\]
which preserves labels; that is: for all $(i,j)$ in $\lambda \setminus
\mu$, the condition $f(i,j) = (h,k)$ implies that $i-j = h-k$.
\end{defn}

The row and column reading sequences (and words) of
equivalent skew shapes are the same.  The following is the main result
of this section. The proof given below relies on results in
\cite{BJS93}.

\begin{thm}\label{t:321}
  Let $\Theta$ be the map $d \mapsto (\lambda \setminus \mu, \omega)$
  defined by Algorithm~\ref{a:skew}. Then:
  \begin{enumerate}
  \item $\Theta$ induces a bijection between the set of $n$-diagrams
    and the set of BJS-equivalence classes of labeled skew shapes
    having numbers all less than $n$.
  \item For any $n$-diagram $d$, the row reading word of $\Theta(d)$
    is the Jones normal form of $d$, in the sense of
    Theorem~\ref{t:Jones}, and the column reading word is the dual
    Jones normal form of $d$, in the sense of Remark~\ref{r:dual-JNF}.
  \end{enumerate}
\end{thm}

\begin{proof}
Recall that a permutation $\pi$ is $321$-avoiding if it never sends
any $i<j<k$ to $\pi(i) > \pi(j) > \pi(k)$. In other words, $\pi$ is
$321$-avoiding if and only if it has no decreasing subsequence of
length three when written in one-line notation. In
\cite{BJS93}*{Thm.~2.1} it is proved that a permutation is
$321$-avoiding if and only if it is \emph{fully commutative} in the
sense of \cite{Stembridge}. (This means that any reduced expression in
terms of the usual Coxeter generators $s_i$ is obtained from any other
by performing a finite sequence of commutations of the form
$s_is_j=s_js_i$ where $|i-j|>1$; the notion generalizes to any Coxeter
group.)

Now consider the quotient map $\HH_n(-v^{-1},v) \to \TL_n(v+v^{-1})$
in Corollary~\ref{c:TL-quo}.  In his dissertation
\cites{Fan:95,Fan:97}, C.K.~Fan proved a more general result that
implies that the image of the set
\[
\{T_w: w \in \Sym_n, \text{ $w$ is $321$-avoiding}\}
\]
under the above map is a basis. (Our quotient map is a renormalization
of his.) This means that we may index $n$-diagrams by $321$-avoiding
permutations. In fact, the bijection
\[
\{\text{321-avoiding permutations in } \Sym_n\} \to
\{n\text{-diagrams}\}
\]
is given by sending any reduced expression for $w$ in terms of the
Coxeter generators $s_i$ to the corresponding reduced expression in
which the $s_i$ are replaced by $e_i$.

That the map in (a) is a bijection now follows from the bijection
\cite{BJS93}*{\S2} between $321$-avoiding permutations and labeled
skew shapes (under BJS-equivalence). Part (b) follows easily once one
notices that the row reading word of the labeled skew shape
$\Theta(d)$ is always in Jones normal form and the column reading word
is always in dual normal form. These claims follow from the fact that
the numbers in $\Theta(d)$ decrease by one along rows and increase by
one down columns.
\end{proof}

\begin{example}
The Jones normal form of the $9$-diagram in Example~\ref{ex:polygonal}
is $d = (e_2e_1)(e_4e_3e_2)(e_8)$ and its dual normal form is $d =
(e_8)(e_4)(e_2e_3)(e_1e_2)$. These are of course the row and column
reading words, respectively, of the corresponding labeled shew shapes.
\end{example}

\begin{rmk}\label{r:321}
(i) Algorithm~\ref{a:skew} also produces reduced expressions for
  $321$-avoiding permutations, thus giving a new proof of the
  bijection in \cite{BJS93}. One simply draws the permutation diagram
  (in the sense of Brauer algebras) and applies the same method. For
  example, consider the permutation $w$ in $\Sym_9$ given by $w =
  351246798$ in the usual one-line notation. It is depicted by the
  Brauer diagram
\[
\begin{tikzpicture}[scale = 0.35,thick, baseline={(0,-1ex/2)}] 
  \tikzstyle{vertex} = [shape = circle, minimum size = 4pt,
    inner sep = 1pt, fill = black] 
\node[vertex] (G--9) at (12.0, -1) [shape = circle, draw] {}; 
\node[vertex] (G-8) at (10.5, 1) [shape = circle, draw] {}; 
\node[vertex] (G--8) at (10.5, -1) [shape = circle, draw] {}; 
\node[vertex] (G-9) at (12.0, 1) [shape = circle, draw] {}; 
\node[vertex] (G--7) at (9.0, -1) [shape = circle, draw] {}; 
\node[vertex] (G-7) at (9.0, 1) [shape = circle, draw] {}; 
\node[vertex] (G--6) at (7.5, -1) [shape = circle, draw] {}; 
\node[vertex] (G-6) at (7.5, 1) [shape = circle, draw] {}; 
\node[vertex] (G--5) at (6.0, -1) [shape = circle, draw] {}; 
\node[vertex] (G-4) at (4.5, 1) [shape = circle, draw] {}; 
\node[vertex] (G--4) at (4.5, -1) [shape = circle, draw] {}; 
\node[vertex] (G-2) at (1.5, 1) [shape = circle, draw] {}; 
\node[vertex] (G--3) at (3.0, -1) [shape = circle, draw] {}; 
\node[vertex] (G-1) at (0.0, 1) [shape = circle, draw] {}; 
\node[vertex] (G--2) at (1.5, -1) [shape = circle, draw] {}; 
\node[vertex] (G-5) at (6.0, 1) [shape = circle, draw] {}; 
\node[vertex] (G--1) at (0.0, -1) [shape = circle, draw] {}; 
\node[vertex] (G-3) at (3.0, 1) [shape = circle, draw] {}; 
\draw[] (G-8) .. controls +(0.75, -1) and +(-0.75, 1) .. (G--9); 
\draw[] (G-9) .. controls +(-0.75, -1) and +(0.75, 1) .. (G--8); 
\draw[] (G-7) .. controls +(0, -1) and +(0, 1) .. (G--7); 
\draw[] (G-6) .. controls +(0, -1) and +(0, 1) .. (G--6); 
\draw[] (G-4) .. controls +(0.75, -1) and +(-0.75, 1) .. (G--5); 
\draw[] (G-2) .. controls +(1, -1) and +(-1, 1) .. (G--4); 
\draw[] (G-1) .. controls +(1, -1) and +(-1, 1) .. (G--3); 
\draw[] (G-5) .. controls +(-1, -1) and +(1, 1) .. (G--2); 
\draw[] (G-3) .. controls +(-1, -1) and +(1, 1) .. (G--1); 
\end{tikzpicture}
\]
in which $w(i)=j$ is depicted by a strand connecting the $i$th vertex
in the bottom row with the $j$th vertex in the top row. Applying the
algorithm computes the same polygonal path that appears in
Example~\ref{ex:polygonal}.  The row and column reading
words of the corresponding labeled skew shape, written in terms of the
$s_i$ generators, are respectively $(s_2s_1)(s_4s_3s_2)(s_8)$ and
$(s_8)(s_4)(s_2s_3)(s_1s_2)$. Both are reduced expressions
for $w$.

(ii) Let $d$ be an $n$-diagram and $\Theta(d) = (\lambda \setminus
\mu, \omega)$ its labeled skew shape. Let $w \in \Sym_n$ be the
corresponding $321$-avoiding permutation. By \cite{BJS93}*{Cor.~2.1},
the number of reduced expressions for $w$ is the number of standard
tableaux of shape $\lambda \setminus \mu$. This of course is also the
number of reduced expressions for $d$ in terms of the Temperley--Lieb
generators $e_i$.

(iii) A pleasant aspect of the mapping from $n$-diagrams to polygonal
paths is that it distinguishes generators in different $\TL_n$. For
instance, the polygonal path of $e_2$ in $\TL_3$ is different from
that of $e_2$ in $\TL_n$, for any $n>3$. (Indeed, the paths are of
different lengths.)  We display the polygonal paths of $e_2$ in
$\TL_3$ and $\TL_4$ respectively below
\[
\begin{tikzpicture}[scale = 0.4,thick, baseline={(0,-1ex/2)}] 
  \draw[very thick]
(0,0)--(0,-1)--(-1,-1)--(-1,-2)--(0,-2)--(0,-1)--(0,0);
  \draw[step=1,dashed,gray,very thin] (-1,-2) grid (0,0);
\end{tikzpicture}\qquad\qquad
\begin{tikzpicture}[scale = 0.4,thick, baseline={(0,-1ex/2)}] 
  \draw[very thick]
(0,0)--(0,-1)--(-1,-1)--(-1,-2)--(-1,-3)--(-1,-2)--(0,-2)--(0,-1)--(0,0);
  \draw[step=1,dashed,gray,very thin] (-2,-3) grid (0,0);
\end{tikzpicture}
\]
in order to illustrate this point. 
\end{rmk}

\section{Representations of $\TL_n$}\noindent
In this section, $\Bbbk$ is a commutative unital ring and $\delta \in
\Bbbk$, unless stated otherwise.
We fix $n$ and $\delta$ and sometimes write $\TL_n = \TL_n(\delta)$.

We begin with a number of bijections that underlie the combinatorics
of Temperley--Lieb algebras. In Section~\ref{s:dia} we considered
lattice walks to $(n-p,p)$, where $0\le 2p \le n$. Notice that the pair
$(n-p,p)$ in such a walk may be identified with a partition of at most
two parts, which in turn may be identified with its Young diagram.

A $1$-\emph{factor} is a sequence $f = (f_1, \dots, f_n)$ such that
each $f_i = \pm 1$ and the partial sums $f_1 + \cdots + f_i$ are
nonnegative, for all $i$. For each $i$ with $f_i=1$ in a $1$-factor
$f$, let $j$ be the smallest index (if any) for which $i < j \le n$
and $f_i + \cdots + f_j = 0$. Whenever this happens, the indices
$(i,j)$ are said to be \emph{paired}; otherwise the index $i$ is
\emph{unpaired}.

The Bratteli diagram associated to Temperley--Lieb combinatorics is
the infinite graph constructed inductively as follows:
\begin{itemize}
\item Start with the empty partition $\emptyset$ in level zero.
\item For each partition $\lambda=(\lambda_1,\lambda_2)$ in some
  level, draw a vertical edge to the partition
  $(\lambda_1+1,\lambda_2)$ and, if $\lambda_1>\lambda_2$, a diagonal
  edge to the partition $(\lambda_1,\lambda_2+1)$.
\end{itemize}
We illustrate the Bratteli diagram in Figure \ref{Bratteli}.
\begin{figure}[ht]
\begin{center}
\begin{tikzpicture}[xscale=6*\UNIT, yscale=-3*\UNIT]
  \coordinate (0) at (0,0);
  \foreach \x in {0,...,0}{\coordinate (1\x) at (2*\x,2);}
  \foreach \x in {0,...,1}{\coordinate (2\x) at (2*\x,4);}
  \foreach \x in {0,...,1}{\coordinate (3\x) at (2*\x,6);}
  \foreach \x in {0,...,2}{\coordinate (4\x) at (2*\x,8);}
  \foreach \x in {0,...,2}{\coordinate (5\x) at (2*\x,10);}
  \foreach \x in {0,...,3}{\coordinate (6\x) at (2*\x,12);}
  \foreach \x in {0,...,3}{\coordinate (7\x) at (2*\x,14);}

  \draw (0)--(10);

  \draw (10)--(20) (10)--(21);

  \draw (20)--(30) (20)--(31) (21)--(31);

  \draw (30)--(40) (30)--(41) (31)--(41) (31)--(42);

  \draw (40)--(50) (40)--(51) (41)--(51) (41)--(52) (42)--(52);

  \draw (50)--(60) (50)--(61) (51)--(61) (51)--(62) (52)--(62) (52)--(63);

  \draw (60)--(70) (60)--(71) (61)--(71) (61)--(72) (62)--(72)
  (62)--(73) (63)--(73);


  
\def\NULL{\footnotesize$\emptyset$}
\begin{scope}[every node/.style={fill=white}]
  \node at (0) {\NULL};

  \node at (10) {\PART{1}};

  \node at (20) {\PART{2}};
  \node at (21) {\PART{1,1}};

  \node at (30) {\PART{3}};
  \node at (31) {\PART{2,1}};

  \node at (40) {\PART{4}};
  \node at (41) {\PART{3,1}};
  \node at (42) {\PART{2,2}};

  \node at (50) {\PART{5}};
  \node at (51) {\PART{4,1}};
  \node at (52) {\PART{3,2}};

  \node at (60) {\PART{6}};
  \node at (61) {\PART{5,1}};
  \node at (62) {\PART{4,2}};
  \node at (63) {\PART{3,3}};

  \node at (70) {\PART{7}};
  \node at (71) {\PART{6,1}};
  \node at (72) {\PART{5,2}};
  \node at (73) {\PART{4,3}};

%
  
\end{scope}
\end{tikzpicture}
\end{center}
\caption{Bratteli diagram up to level $7$}\label{Bratteli}\label{f:Bratteli}
\end{figure}
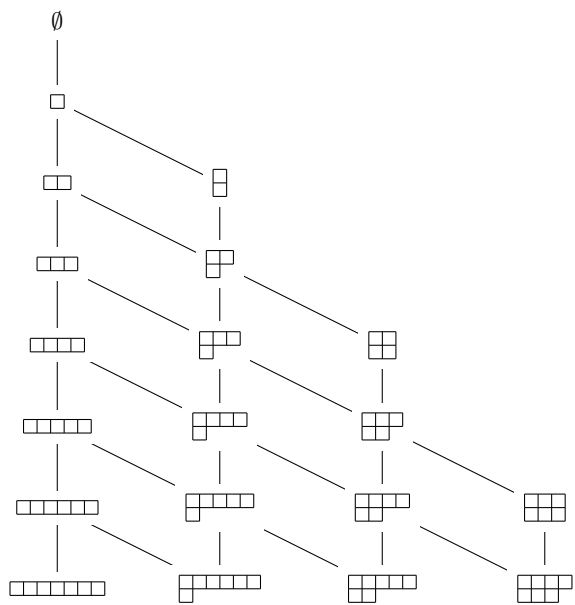

Here are the promised bijections.

\begin{lem}\label{l:bijections}
  For any $n$, $p$ such that $0 \le 2p \le n$, the following sets are
  all in bijective correspondence with one another:
  \begin{enumerate}\renewcommand{\labelenumi}{(\roman{enumi})}
  \item The set of half-diagrams on $n$ vertices with $p$ links.
  \item The set of lattice walks from $(0,0)$ to $(n-p,p)$.
  \item The set of paths in the Bratteli diagram
    from $\emptyset$ to $(n-p,p)$.
  \item The set of standard tableaux of shape $(n-p,p)$.
  \item The set of $1$-factors of length $n$ with $p$ pairings.
  \end{enumerate}
\end{lem}

\begin{proof}
The bijection between the sets in (i), (ii) is Lemma
\ref{l:half-count}.  The bijection between the sets in (ii), (iii) is
obtained by matching (horizontal, vertical) segments in a lattice walk
with (vertical, diagonal) edges in a Bratteli path. The bijection
between the sets in (ii), (iv) comes from numbering each unit-length
segment in a lattice walk, in order. Write the numbers into the boxes
of a Young diagram of shape $(n-p,p)$ so that horizontal segments are
recorded in row one, and vertical segments in row two.  For instance,
the tableau
\[
\ytableaushort{12358,467} 
\]
corresponds to the lattice walk appearing in the proof of
Lemma~\ref{l:half-count}.  Note that the numbers are entered in order
in each row from left to right; this always produces a standard
tableau. Finally, a bijection between the sets in (i), (v) is easily
obtained by matching links with paired vertices and defects with
unpaired ones. 
\end{proof}

In this section, we prefer to use the set of half-diagrams from part
(i) of Lemma~\ref{l:bijections}, but that indexing set may be replaced
by any of the others. We note that half-diagrams are called ``planar
involutions'' in \cite{GL:96}. We need the following notation.  Set
\[
\Lambda = \Lambda(n) = \{n, n-2, \dots, n-2l\}
\]
where $l$ is the integer part of $n/2$. Notice that the map $n-2p
\mapsto (n-p,p)$ for $0\le 2p \le n$ defines a bijection between
$\Lambda$ and the set of two-part partitions of $n$. For each $\lambda
\in \Lambda$, let
\[
M(\lambda) = \text{ the set of half-diagrams on $n$ vertices with
  $\lambda$ defects}.
\]
Given any $(s,t) \in M(\lambda)\times M(\lambda)$, let $t^*$ be the
reflection of $t$ across the line containing its vertices. Place $s$
directly above $t^*$. There is one and only one way to
connect the defects in $s$ to the defects in $t^*$ so as to make an
$n$-diagram. Let
\[
C^\lambda_{s,t} = \text{ the $n$-diagram obtained by this process.}
\]
Then the disjoint union $\bigsqcup_{\lambda\in \Lambda}
\{C^\lambda_{s,t} \mid s,t \in M(\lambda)\}$ is the basis consisting
of all $n$-diagrams.  Finally, let
\[
*: \TL_n \to \TL_n
\]
be the linear extension of the map that reflects a given diagram
across its axis of symmetry with respect to the parallel lines
determined by its vertices. Write $d^*$ for the image of a diagram $d$
under this map. Then
\[
d^{**}=d \quad\text{and}\quad (d_1d_2)^* = d_2^*d_1^*
\]
for all $n$-diagrams $d_1$, $d_2$. In other words, the map $*$ is an
algebra anti-involution of $\TL_n$. We have $(C^\lambda_{s,t})^* =
C^\lambda_{t,s}$ for all $s,t \in M(\lambda)$, $\lambda \in \Lambda$.
By Lemma~\ref{l:bijections} and equation~\eqref{e:Cnp}, the
cardinality of $M(\lambda)$ is given by
\begin{equation}\label{e:cardMlam}
  |M(\lambda)| = \Cat_{n,p} = \binom{n}{p} - \binom{n}{p-1} \qquad
  \text{if } n-2p = \lambda
\end{equation}
for each $\lambda$ in $\Lambda$. It is easy to check the following
(see \cite{GL:96}*{Example 1.4}).

\begin{prop}
  Let $\Bbbk$ be a unital commutative ring, $\delta \in \Bbbk$. Then
  the datum $(\Lambda,M,C,*)$ defined above is a cell datum for the
  algebra $\TL_n = \TL_n(\delta)$, in the sense of \cite{GL:96}. In
  other words, $\TL_n$ is cellular, and the basis of $n$-diagrams is a
  cellular basis.
\end{prop}

We now construct some $\TL_n$-modules diagrammatically.  If
$h$ is a half-diagram on $n$ vertices and $d$ an $n$-diagram, we stack
$d$ above $h$ and apply the diagrammatic multiplication rule
\eqref{e:mult-rule} to obtain
\begin{equation}\label{e:half-m-rule}
  d h = \delta^N \, h'
\end{equation}
for a unique half-diagram $h'$ (obtained by discarding the loops and
identified vertices and retaining links) and some integer $N \ge 0$
(the number of discarded loops). 
For example,
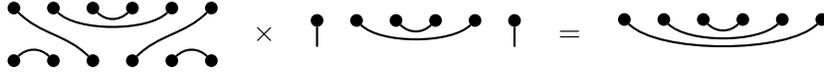
\[
\begin{tikzpicture}[scale = 0.35,thick, baseline={(0,-1ex/2)}] 
\tikzstyle{vertex} = [shape = circle, minimum size = 4pt,
    inner sep = 1pt, fill=black] 
\node[vertex] (G--6) at (7.5, -1) [shape = circle, draw] {}; 
\node[vertex] (G--5) at (6.0, -1) [shape = circle, draw] {}; 
\node[vertex] (G--4) at (4.5, -1) [shape = circle, draw] {}; 
\node[vertex] (G-6) at (7.5, 1) [shape = circle, draw] {}; 
\node[vertex] (G--3) at (3.0, -1) [shape = circle, draw] {}; 
\node[vertex] (G-1) at (0.0, 1) [shape = circle, draw] {}; 
\node[vertex] (G--2) at (1.5, -1) [shape = circle, draw] {}; 
\node[vertex] (G--1) at (0.0, -1) [shape = circle, draw] {}; 
\node[vertex] (G-2) at (1.5, 1) [shape = circle, draw] {}; 
\node[vertex] (G-5) at (6.0, 1) [shape = circle, draw] {}; 
\node[vertex] (G-3) at (3.0, 1) [shape = circle, draw] {}; 
\node[vertex] (G-4) at (4.5, 1) [shape = circle, draw] {}; 
\draw[] (G--6) .. controls +(-0.5, 0.5) and +(0.5, 0.5) .. (G--5); 
\draw[] (G-6) .. controls +(-1, -1) and +(1, 1) .. (G--4); 
\draw[] (G-1) .. controls +(1, -1) and +(-1, 1) .. (G--3); 
\draw[] (G--2) .. controls +(-0.5, 0.5) and +(0.5, 0.5) .. (G--1); 
\draw[] (G-2) .. controls +(0.9, -0.9) and +(-0.9, -0.9) .. (G-5); 
\draw[] (G-3) .. controls +(0.5, -0.5) and +(-0.5, -0.5) .. (G-4); 
\end{tikzpicture} \quad \times \quad
\begin{minipage}{2.8cm}
\begin{tikzpicture}[scale = 0.35,thick, baseline={(0,-1ex/2)}] 
\tikzstyle{vertex} = [shape = circle, minimum size = 4pt,
    inner sep = 1pt, fill=black] 
\node[vertex] (G-6) at (7.5, 1) [shape = circle, draw] {}; 
\node[vertex] (G-1) at (0.0, 1) [shape = circle, draw] {}; 
\node[vertex] (G-2) at (1.5, 1) [shape = circle, draw] {}; 
\node[vertex] (G-5) at (6.0, 1) [shape = circle, draw] {}; 
\node[vertex] (G-3) at (3.0, 1) [shape = circle, draw] {}; 
\node[vertex] (G-4) at (4.5, 1) [shape = circle, draw] {};
\draw[] (G-6) -- (7.5,0);
\draw[] (G-1) -- (0,0);
\draw[] (G-2) .. controls +(0.9, -0.9) and +(-0.9, -0.9) .. (G-5); 
\draw[] (G-3) .. controls +(0.5, -0.5) and +(-0.5, -0.5) .. (G-4); 
\end{tikzpicture}
\end{minipage} \quad = \quad
\begin{minipage}{2.8cm}
\begin{tikzpicture}[scale = 0.35,thick, baseline={(0,-1ex/2)}] 
\tikzstyle{vertex} = [shape = circle, minimum size = 4pt,
    inner sep = 1pt, fill=black] 
\node[vertex] (G-6) at (7.5, 1) [shape = circle, draw] {}; 
\node[vertex] (G-1) at (0.0, 1) [shape = circle, draw] {}; 
\node[vertex] (G-2) at (1.5, 1) [shape = circle, draw] {}; 
\node[vertex] (G-5) at (6.0, 1) [shape = circle, draw] {}; 
\node[vertex] (G-3) at (3.0, 1) [shape = circle, draw] {}; 
\node[vertex] (G-4) at (4.5, 1) [shape = circle, draw] {};
\draw[] (G-1) .. controls +(1.3, -1.3) and +(-1.3, -1.3) .. (G-6);
\draw[] (G-2) .. controls +(0.9, -0.9) and +(-0.9, -0.9) .. (G-5); 
\draw[] (G-3) .. controls +(0.5, -0.5) and +(-0.5, -0.5) .. (G-4); 
\end{tikzpicture}
\end{minipage} 
\]
as one can see by considering the configuration 
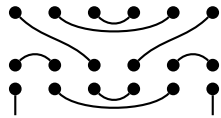
\[
\begin{gathered}
\begin{tikzpicture}[scale = 0.35,thick, baseline={(0,-1ex/2)}] 
\tikzstyle{vertex} = [shape = circle, minimum size = 4pt,
    inner sep = 1pt, fill=black] 
\node[vertex] (G--6) at (7.5, -1) [shape = circle, draw] {}; 
\node[vertex] (G--5) at (6.0, -1) [shape = circle, draw] {}; 
\node[vertex] (G--4) at (4.5, -1) [shape = circle, draw] {}; 
\node[vertex] (G-6) at (7.5, 1) [shape = circle, draw] {}; 
\node[vertex] (G--3) at (3.0, -1) [shape = circle, draw] {}; 
\node[vertex] (G-1) at (0.0, 1) [shape = circle, draw] {}; 
\node[vertex] (G--2) at (1.5, -1) [shape = circle, draw] {}; 
\node[vertex] (G--1) at (0.0, -1) [shape = circle, draw] {}; 
\node[vertex] (G-2) at (1.5, 1) [shape = circle, draw] {}; 
\node[vertex] (G-5) at (6.0, 1) [shape = circle, draw] {}; 
\node[vertex] (G-3) at (3.0, 1) [shape = circle, draw] {}; 
\node[vertex] (G-4) at (4.5, 1) [shape = circle, draw] {}; 
\draw[] (G--6) .. controls +(-0.5, 0.5) and +(0.5, 0.5) .. (G--5); 
\draw[] (G-6) .. controls +(-1, -1) and +(1, 1) .. (G--4); 
\draw[] (G-1) .. controls +(1, -1) and +(-1, 1) .. (G--3); 
\draw[] (G--2) .. controls +(-0.5, 0.5) and +(0.5, 0.5) .. (G--1); 
\draw[] (G-2) .. controls +(0.9, -0.9) and +(-0.9, -0.9) .. (G-5); 
\draw[] (G-3) .. controls +(0.5, -0.5) and +(-0.5, -0.5) .. (G-4); 
\end{tikzpicture}\\
\begin{tikzpicture}[scale = 0.35,thick, baseline={(0,-1ex/2)}] 
\tikzstyle{vertex} = [shape = circle, minimum size = 4pt,
    inner sep = 1pt, fill=black] 
\node[vertex] (G-6) at (7.5, 1) [shape = circle, draw] {}; 
\node[vertex] (G-1) at (0.0, 1) [shape = circle, draw] {}; 
\node[vertex] (G-2) at (1.5, 1) [shape = circle, draw] {}; 
\node[vertex] (G-5) at (6.0, 1) [shape = circle, draw] {}; 
\node[vertex] (G-3) at (3.0, 1) [shape = circle, draw] {}; 
\node[vertex] (G-4) at (4.5, 1) [shape = circle, draw] {};
\draw[] (G-6) -- (7.5,0);
\draw[] (G-1) -- (0,0);
\draw[] (G-2) .. controls +(0.9, -0.9) and +(-0.9, -0.9) .. (G-5); 
\draw[] (G-3) .. controls +(0.5, -0.5) and +(-0.5, -0.5) .. (G-4); 
\end{tikzpicture}
\end{gathered}
\]
obtained by the usual stacking procedure. This example shows,
incidentally, that the action does not always preserve the number of
defects, although the number of defects in $h'$ cannot exceed the
number in $h$.

Let $\hat{H}$ be the $\Bbbk$-linear span of the set
$\bigsqcup_{\lambda \in \Lambda} M(\lambda)$. That is, $\hat{H}$ is
the span of the set of half-diagrams on $n$ vertices. The linear
extension of the action defined in \eqref{e:half-m-rule} makes
$\hat{H}$ into a $\TL_n(\delta)$-module. For each $\lambda$ in
$\Lambda$, let
\[
\hat{H}^{\le \lambda} = \text{ $\Bbbk$-span of } \textstyle
  \bigsqcup_{\mu \le \lambda} M(\mu), \quad \hat{H}^{< \lambda} =
\text{ $\Bbbk$-span of } \textstyle \bigsqcup_{\mu < \lambda}
  M(\mu)
\]
where $\mu$ ranges over $\Lambda$ in both unions.  Since the
$\TL_n$-action cannot increase the number of defects, these spans are
$\TL_n$-submodules of $\hat{H}$. For each $\lambda \in \Lambda$, set
\[
H(\lambda) := \hat{H}^{\le \lambda}/\hat{H}^{< \lambda}. 
\]
A basis for $H(\lambda)$ is $\{h+\hat{H}^{< \lambda} \mid h \in
M(\lambda)\}$.  If we abuse notation by denoting a coset $h+\hat{H}^{<
  \lambda}$ by its chosen representative $h$, for any $h \in
M(\lambda)$, then the action of an $n$-diagram $d$ on $h$ is given by:
\begin{equation}\label{e:quo-act}
d h = 
\begin{cases}
  \delta^N \, h' & \text{ if } h' \in M(\lambda)\\ 0 & \text{ otherwise}
\end{cases}
\end{equation}
with $N$, $h'$ as in \eqref{e:half-m-rule}.
With this convention, the set $M(\lambda)$ is a basis of $H(\lambda)$
and the rule \eqref{e:quo-act} defines its $\TL_n(\delta)$-module
structure.

Graham and Lehrer prove that any cellular algebra has an
associated family of \emph{cell modules}, which may be constructed
abstractly from its chosen cellular basis.

\begin{prop}\label{p:Hm}
  Let $\delta$ be an element of a unital commutative ring $\Bbbk$.
  For any $\lambda \in \Lambda$, the module $H(\lambda)$ is isomorphic
  to the abstract cell module indexed by $\lambda$ in the theory of
  cellular algebras. 
\end{prop}

A proof of this fact is implicit in \cite{GL:96}. Alternatively, the
reader may prefer to construct the needed isomorphism directly, which
is not difficult.  Note that the rank of $H(\lambda)$ over $\Bbbk$ is
computed in equation \eqref{e:cardMlam}.

We take a moment to consider some general facts on cellular algebras.
If $A$ is cellular with cell datum $(\Lambda,M,C,*)$, let $W(\lambda)$
be the cell module indexed by $\lambda$ in $\Lambda$. There is an
associated bilinear form $\varphi_\lambda$ on $W(\lambda)$. Let
$\Lambda_0 := \{\lambda \in \Lambda \mid \varphi_\lambda \ne 0\}$.  By
\cite{GL:96}, if the ground ring $\Bbbk$ is a field then the collection
\[
  \{ L(\lambda):= W(\lambda)/\text{rad}(\varphi) \mid \lambda \in
  \Lambda_0 \}
\]
gives a complete set, up to isomorphism, of simple
$A$-modules. Furthermore, each simple module is absolutely simple.
Finally, $A$ is (split) semisimple over a field if and only if all of
its cell modules are simple.

Return now to the consideration of $A=\TL_n$. Combining the final
sentence of the preceding paragraph with Corollary \ref{c:ss-crit}
gives the following.

\begin{prop}
  Let $\Bbbk$ be a field, and suppose that $0 \ne v \in \Bbbk$ such
  that $[n]_v^! \ne 0$. If $\delta = \pm(v+v^{-1}) \ne 0$ then
  $\{H(\lambda)\mid \lambda \in \Lambda\}$ is a complete set of simple
  $\TL_n(\delta)$-modules, up to isomorphism.
\end{prop}

Let $\varphi_\lambda(-,-)$ be the bilinear form associated to
$H(\lambda)$.  We claim that $\varphi_\lambda$ may be computed
diagrammatically.  If $h$ is a half-diagram, let $h^*$ be the result
of reflecting $h$ across the line containing its vertices.  Given $h,
h' \in M(\lambda)$, where $\lambda \in \Lambda$, let $h^*|h'$ be the
configuration obtained by stacking $h^*$ above $h'$. We say that
$h^*|h'$ is \emph{defect preserving} if every defect in one of the
half-diagrams is connected by a path to a defect in the other
half-diagram, after corresponding vertices are identified. Then
$\varphi_\lambda(h,h')$ is given by
\begin{equation}
  \varphi_\lambda(h,h') = 
  \begin{cases}
    \delta^N & \text{if $h^*|h'$ is defect preserving}\\
    0 & \text{otherwise}
  \end{cases}
\end{equation}
where $N$ is the number of loops in $h^*|h'$ (after corresponding
vertices are identified). The form $\varphi_\lambda$ is associative:
$\varphi_\lambda(th,h') = \varphi_\lambda(h,t^*h')$, for any $t$ in
$\TL_n(\delta)$.

\begin{example}
If $n=6$ then we have the following, which illustrate the various
cases that can occur.
\begin{enumerate}\renewcommand{\labelenumi}{(\roman{enumi})}
\item \quad $\varphi_0\big(
\begin{tikzpicture}[scale = 0.35,thick, baseline={(0,1ex)}] 
\tikzstyle{vertex} = [shape = circle, minimum size = 4pt,
    inner sep = 1pt, fill=black] 
\node[vertex] (G-6) at (7.5, 1) [shape = circle, draw] {}; 
\node[vertex] (G-1) at (0.0, 1) [shape = circle, draw] {}; 
\node[vertex] (G-2) at (1.5, 1) [shape = circle, draw] {}; 
\node[vertex] (G-5) at (6.0, 1) [shape = circle, draw] {}; 
\node[vertex] (G-3) at (3.0, 1) [shape = circle, draw] {}; 
\node[vertex] (G-4) at (4.5, 1) [shape = circle, draw] {};
\draw[] (G-5) .. controls +(0.5, -0.5) and +(-0.5, -0.5) .. (G-6);
\draw[] (G-1) .. controls +(0.9, -0.9) and +(-0.9, -0.9) .. (G-4); 
\draw[] (G-2) .. controls +(0.5, -0.5) and +(-0.5, -0.5) .. (G-3); 
\end{tikzpicture}
\; , \;
\begin{tikzpicture}[scale = 0.35,thick, baseline={(0,1ex)}] 
\tikzstyle{vertex} = [shape = circle, minimum size = 4pt,
    inner sep = 1pt, fill=black] 
\node[vertex] (G-6) at (7.5, 1) [shape = circle, draw] {}; 
\node[vertex] (G-1) at (0.0, 1) [shape = circle, draw] {}; 
\node[vertex] (G-2) at (1.5, 1) [shape = circle, draw] {}; 
\node[vertex] (G-5) at (6.0, 1) [shape = circle, draw] {}; 
\node[vertex] (G-3) at (3.0, 1) [shape = circle, draw] {}; 
\node[vertex] (G-4) at (4.5, 1) [shape = circle, draw] {};
\draw[] (G-5) .. controls +(0.5, -0.5) and +(-0.5, -0.5) .. (G-6);
\draw[] (G-3) .. controls +(0.5, -0.5) and +(-0.5, -0.5) .. (G-4); 
\draw[] (G-1) .. controls +(0.5, -0.5) and +(-0.5, -0.5) .. (G-2); 
\end{tikzpicture}
\big) \; = \; \delta^2$.

\item \quad $\varphi_2\big(
\begin{tikzpicture}[scale = 0.35,thick, baseline={(0,1ex)}] 
\tikzstyle{vertex} = [shape = circle, minimum size = 4pt,
    inner sep = 1pt, fill=black] 
\node[vertex] (G-6) at (7.5, 1) [shape = circle, draw] {}; 
\node[vertex] (G-1) at (0.0, 1) [shape = circle, draw] {}; 
\node[vertex] (G-2) at (1.5, 1) [shape = circle, draw] {}; 
\node[vertex] (G-5) at (6.0, 1) [shape = circle, draw] {}; 
\node[vertex] (G-3) at (3.0, 1) [shape = circle, draw] {}; 
\node[vertex] (G-4) at (4.5, 1) [shape = circle, draw] {};
\draw[] (G-1) .. controls +(0.5, -0.5) and +(-0.5, -0.5) .. (G-2);
\draw[] (G-3) -- (3.0,0);
\draw[] (G-6) -- (7.5,0);
\draw[] (G-4) .. controls +(0.5, -0.5) and +(-0.5, -0.5) .. (G-5); 
\end{tikzpicture}
\; , \;
\begin{tikzpicture}[scale = 0.35,thick, baseline={(0,1ex)}] 
\tikzstyle{vertex} = [shape = circle, minimum size = 4pt,
    inner sep = 1pt, fill=black] 
\node[vertex] (G-6) at (7.5, 1) [shape = circle, draw] {}; 
\node[vertex] (G-1) at (0.0, 1) [shape = circle, draw] {}; 
\node[vertex] (G-2) at (1.5, 1) [shape = circle, draw] {}; 
\node[vertex] (G-5) at (6.0, 1) [shape = circle, draw] {}; 
\node[vertex] (G-3) at (3.0, 1) [shape = circle, draw] {}; 
\node[vertex] (G-4) at (4.5, 1) [shape = circle, draw] {};
\draw[] (G-5) -- (6.0,0);
\draw[] (G-6) -- (7.5,0);
\draw[] (G-3) .. controls +(0.5, -0.5) and +(-0.5, -0.5) .. (G-4); 
\draw[] (G-1) .. controls +(0.5, -0.5) and +(-0.5, -0.5) .. (G-2); 
\end{tikzpicture}
\big) \; = \; \delta$.

\item \quad $\varphi_2\big(
\begin{tikzpicture}[scale = 0.35,thick, baseline={(0,1ex)}] 
\tikzstyle{vertex} = [shape = circle, minimum size = 4pt,
    inner sep = 1pt, fill=black] 
\node[vertex] (G-6) at (7.5, 1) [shape = circle, draw] {}; 
\node[vertex] (G-1) at (0.0, 1) [shape = circle, draw] {}; 
\node[vertex] (G-2) at (1.5, 1) [shape = circle, draw] {}; 
\node[vertex] (G-5) at (6.0, 1) [shape = circle, draw] {}; 
\node[vertex] (G-3) at (3.0, 1) [shape = circle, draw] {}; 
\node[vertex] (G-4) at (4.5, 1) [shape = circle, draw] {};
\draw[] (G-1) .. controls +(0.5, -0.5) and +(-0.5, -0.5) .. (G-2);
\draw[] (G-3) -- (3.0,0);
\draw[] (G-4) -- (4.5,0);
\draw[] (G-5) .. controls +(0.5, -0.5) and +(-0.5, -0.5) .. (G-6); 
\end{tikzpicture}
\; , \;
\begin{tikzpicture}[scale = 0.35,thick, baseline={(0,1ex)}] 
\tikzstyle{vertex} = [shape = circle, minimum size = 4pt,
    inner sep = 1pt, fill=black] 
\node[vertex] (G-6) at (7.5, 1) [shape = circle, draw] {}; 
\node[vertex] (G-1) at (0.0, 1) [shape = circle, draw] {}; 
\node[vertex] (G-2) at (1.5, 1) [shape = circle, draw] {}; 
\node[vertex] (G-5) at (6.0, 1) [shape = circle, draw] {}; 
\node[vertex] (G-3) at (3.0, 1) [shape = circle, draw] {}; 
\node[vertex] (G-4) at (4.5, 1) [shape = circle, draw] {};
\draw[] (G-5) -- (6.0,0);
\draw[] (G-6) -- (7.5,0);
\draw[] (G-3) .. controls +(0.5, -0.5) and +(-0.5, -0.5) .. (G-4); 
\draw[] (G-1) .. controls +(0.5, -0.5) and +(-0.5, -0.5) .. (G-2); 
\end{tikzpicture}
\big) \; = \; 0$.
\end{enumerate}
One sees this by looking respectively at the three stack
configurations $h^*|h'$
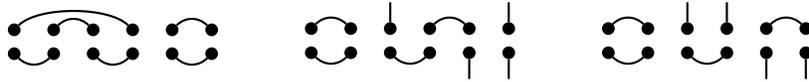
\[
\begin{gathered}
\begin{tikzpicture}[scale = 0.35,thick, baseline={(0,1ex)}] 
\tikzstyle{vertex} = [shape = circle, minimum size = 4pt,
    inner sep = 1pt, fill=black] 
\node[vertex] (G-6) at (7.5, 0) [shape = circle, draw] {}; 
\node[vertex] (G-1) at (0.0, 0) [shape = circle, draw] {}; 
\node[vertex] (G-2) at (1.5, 0) [shape = circle, draw] {}; 
\node[vertex] (G-5) at (6.0, 0) [shape = circle, draw] {}; 
\node[vertex] (G-3) at (3.0, 0) [shape = circle, draw] {}; 
\node[vertex] (G-4) at (4.5, 0) [shape = circle, draw] {};
\draw[] (G-6) .. controls +(-0.5, 0.5) and +(0.5, 0.5) .. (G-5);
\draw[] (G-4) .. controls +(-0.9, 0.9) and +(0.9, 0.9) .. (G-1); 
\draw[] (G-3) .. controls +(-0.5, 0.5) and +(0.5, 0.5) .. (G-2); 
\end{tikzpicture}
  \\
\begin{tikzpicture}[scale = 0.35,thick, baseline={(0,-1ex)}] 
\tikzstyle{vertex} = [shape = circle, minimum size = 4pt,
    inner sep = 1pt, fill=black] 
\node[vertex] (G-6) at (7.5, 0) [shape = circle, draw] {}; 
\node[vertex] (G-1) at (0.0, 0) [shape = circle, draw] {}; 
\node[vertex] (G-2) at (1.5, 0) [shape = circle, draw] {}; 
\node[vertex] (G-5) at (6.0, 0) [shape = circle, draw] {}; 
\node[vertex] (G-3) at (3.0, 0) [shape = circle, draw] {}; 
\node[vertex] (G-4) at (4.5, 0) [shape = circle, draw] {};
\draw[] (G-5) .. controls +(0.5, -0.5) and +(-0.5, -0.5) .. (G-6);
\draw[] (G-3) .. controls +(0.5, -0.5) and +(-0.5, -0.5) .. (G-4); 
\draw[] (G-1) .. controls +(0.5, -0.5) and +(-0.5, -0.5) .. (G-2); 
\end{tikzpicture}
\end{gathered}
\qquad \quad
\begin{gathered}
\begin{tikzpicture}[scale = 0.35,thick, baseline={(0,1ex)}] 
\tikzstyle{vertex} = [shape = circle, minimum size = 4pt,
    inner sep = 1pt, fill=black] 
\node[vertex] (G-6) at (7.5, 0) [shape = circle, draw] {}; 
\node[vertex] (G-1) at (0.0, 0) [shape = circle, draw] {}; 
\node[vertex] (G-2) at (1.5, 0) [shape = circle, draw] {}; 
\node[vertex] (G-5) at (6.0, 0) [shape = circle, draw] {}; 
\node[vertex] (G-3) at (3.0, 0) [shape = circle, draw] {}; 
\node[vertex] (G-4) at (4.5, 0) [shape = circle, draw] {};
\draw[] (G-2) .. controls +(-0.5, 0.5) and +(0.5, 0.5) .. (G-1);
\draw[] (G-5) .. controls +(-0.5, 0.5) and +(0.5, 0.5) .. (G-4); 
\draw[] (G-3) -- (3.0,1);
\draw[] (G-6) -- (7.5,1);
\end{tikzpicture}
  \\
\begin{tikzpicture}[scale = 0.35,thick, baseline={(0,-1ex)}] 
\tikzstyle{vertex} = [shape = circle, minimum size = 4pt,
    inner sep = 1pt, fill=black] 
\node[vertex] (G-6) at (7.5, 0) [shape = circle, draw] {}; 
\node[vertex] (G-1) at (0.0, 0) [shape = circle, draw] {}; 
\node[vertex] (G-2) at (1.5, 0) [shape = circle, draw] {}; 
\node[vertex] (G-5) at (6.0, 0) [shape = circle, draw] {}; 
\node[vertex] (G-3) at (3.0, 0) [shape = circle, draw] {}; 
\node[vertex] (G-4) at (4.5, 0) [shape = circle, draw] {};
\draw[] (G-5) -- (6.0,-1);
\draw[] (G-6) -- (7.5,-1);
\draw[] (G-3) .. controls +(0.5, -0.5) and +(-0.5, -0.5) .. (G-4); 
\draw[] (G-1) .. controls +(0.5, -0.5) and +(-0.5, -0.5) .. (G-2); 
\end{tikzpicture}
\end{gathered}
\qquad \quad
\begin{gathered}
\begin{tikzpicture}[scale = 0.35,thick, baseline={(0,1ex)}] 
\tikzstyle{vertex} = [shape = circle, minimum size = 4pt,
    inner sep = 1pt, fill=black] 
\node[vertex] (G-6) at (7.5, 0) [shape = circle, draw] {}; 
\node[vertex] (G-1) at (0.0, 0) [shape = circle, draw] {}; 
\node[vertex] (G-2) at (1.5, 0) [shape = circle, draw] {}; 
\node[vertex] (G-5) at (6.0, 0) [shape = circle, draw] {}; 
\node[vertex] (G-3) at (3.0, 0) [shape = circle, draw] {}; 
\node[vertex] (G-4) at (4.5, 0) [shape = circle, draw] {};
\draw[] (G-6) .. controls +(-0.5, 0.5) and +(0.5, 0.5) .. (G-5);
\draw[] (G-4) -- (4.5,1);
\draw[] (G-3) -- (3.0,1);
\draw[] (G-2) .. controls +(-0.5, 0.5) and +(0.5, 0.5) .. (G-1); 
\end{tikzpicture}
  \\
\begin{tikzpicture}[scale = 0.35,thick, baseline={(0,-1ex)}] 
\tikzstyle{vertex} = [shape = circle, minimum size = 4pt,
    inner sep = 1pt, fill=black] 
\node[vertex] (G-6) at (7.5, 0) [shape = circle, draw] {}; 
\node[vertex] (G-1) at (0.0, 0) [shape = circle, draw] {}; 
\node[vertex] (G-2) at (1.5, 0) [shape = circle, draw] {}; 
\node[vertex] (G-5) at (6.0, 0) [shape = circle, draw] {}; 
\node[vertex] (G-3) at (3.0, 0) [shape = circle, draw] {}; 
\node[vertex] (G-4) at (4.5, 0) [shape = circle, draw] {};
\draw[] (G-5) -- (6.0,-1);
\draw[] (G-6) -- (7.5,-1);
\draw[] (G-3) .. controls +(0.5, -0.5) and +(-0.5, -0.5) .. (G-4); 
\draw[] (G-1) .. controls +(0.5, -0.5) and +(-0.5, -0.5) .. (G-2); 
\end{tikzpicture}
\end{gathered}
\]
depicted above. Notice that the first two configurations are defect
preserving, but the third is not.
\end{example}

If $n>0$ is even, the bilinear form $\varphi_0$ satisfies a special
property: $\varphi_0(h,h')$ is a positive power of $\delta$, for any
$h,h' \in M(0)$. Hence, if $\delta=0$ then $\varphi_0=0$. Further
analysis reveals the following.

\begin{thm}[\cite{GL:96}*{Cor.~(6.8)}]
  Let $\delta \in \Bbbk$ where $\Bbbk$ is a field. Then
  \[
  \Lambda_0 =
  \begin{cases}
    \Lambda \setminus \{0\} & \text{ if $n>0$ is even and
      $\delta=0$}, \\
    \Lambda & \text{ otherwise}
  \end{cases}
  \]
  gives the indexing set for the isomorphism classes of simple
  $\TL_n(\delta)$-modules.
\end{thm}

In particular, this shows that $\TL_n(0)$ is not semisimple over a
field whenever $n>0$ is even. The representation theory of $\TL_n(0)$
over $\C$ is of great interest in mathematical physics.

Corollary~\ref{c:ss-crit} gave a sufficient condition for
semisimplicity of $\TL_n(\delta)$ in the case when $\delta \ne 0$. The
more precise classication result is as follows.

\begin{thm}\label{t:ss-gen}
  Let $\Bbbk$ be a field, and fix $0 \ne v \in \Bbbk$. Set $\delta =
  \pm(v+v^{-1})$. Then:
  \begin{enumerate}
  \item If $\delta \ne 0$ then $\TL_n(\delta)$ is semisimple
    if and only if $[n]^!_v \ne 0$ in $\Bbbk$.
  \item If $\delta = 0$ then $\TL_n(0)$ is semisimple if and
    only if
    \[
    \begin{cases}
      \text{$n$ is odd} & \text{ if $\Bbbk$ has characteristic $0$}\\
      \text{$n \in \{1,3, \dots, 2p-1\}$} & \text{ if
        $\Bbbk$ has characteristic $p>0$}.
    \end{cases}
    \]
  \end{enumerate}
\end{thm}

Part (a) goes back to \cite{Westbury}, while part (b) was proved in
\cite{Martin}; see also \cite{Ridout-StAubin}.  An easy but somewhat
more sophisticated recent proof (depending on Schur--Weyl duality and
the theory of tilting modules) covering all cases is given in
\cite{AST}*{Prop.~5.1}.

\begin{rmk}
The condition $[n]^!_v \ne 0$ in part (a) of Theorem~\ref{t:ss-gen} is
satisfied if and only if either:
\begin{enumerate}
\item [(i)] $v^2 \ne 1$ and if $v^2$ is an $r$th root of unity then $r>n$, or

\item [(ii)] $v^2 = 1$ and the characteristic of $\Bbbk$ is strictly greater
than $n$.
\end{enumerate}
%
\end{rmk}

When $\TL_n(\delta)$ isn't semisimple, its representations over a
field have been understood for a long time. The blocks are known, and
the structure of the indecomposable projective modules are known as
well \cites{GHJ,Martin,GW}; see also \cite{Ridout-StAubin}.

\section{Schur--Weyl duality}\label{s:SWD}\noindent
Let $\Bbbk$ be a field in this section. Fix $0 \ne v$ in $\Bbbk$.  We
consider tensor space $V^{\otimes n}$, where $V=V(1)$ is the
$2$-dimensional ``vector'' representation of $\UU(\gl_2)$. By
restriction, $V^{\otimes n}$ is a representation of $\UU(\fsl_2)$. In
this section, we will show that if $\delta = \pm(v+v^{-1})$ then
$\TL_n(\delta) \cong \End_\UU(V^{\otimes n})$ for $\UU = \UU(\gl_2)$
or $\UU(\fsl_2)$. More generally, we show that $V^{\otimes n}$
satisfies Schur--Weyl duality with respect to the commuting actions of
$\UU$ and $\TL_n(\delta)$, where again $\UU = \UU(\gl_2)$ or
$\UU(\fsl_2)$.  The discussion breaks into cases, depending whether
$v$ is or is not a root of unity; we consider the latter case first.

\begin{rmk}
The fact that $\TL_n(\delta)$ is obtained as the centralizer of either
$\UU(\gl_2)$ or $\UU(\fsl_2)$ is a special feature of the vector
representation.  In other closely related situations, for instance if
$V$ is replaced by $V \oplus \Bbbk$, as in \cites{BH:Motzkin,DG:PTL},
the algebras $\End_{\UU(\fsl_2)}((V \oplus \Bbbk)^{\otimes n})$ and
$\End_{\UU(\gl_2)}((V \oplus \Bbbk)^{\otimes n})$ are very different
(indeed, they usually have different dimensions) even if $v$ is not a
root of unity. The former centralizer is the Motzkin algebra of
\cite{BH:Motzkin} while the latter is the partial Temperley--Lieb
algebra of \cite{DG:PTL}.
\end{rmk}

We now recall the definition of $\UU(\gl_2)$, assuming that $v$ is not
a root of unity in $\Bbbk$.
This is the associative $\Bbbk$-algebra with $1$ generated by symbols
$E$, $F$, $K_i^{\pm 1}$ ($i=1,2$) subject to the defining relations
\begin{subequations}\label{e:UU-gl}
\begin{alignat}{-1}
  & K_1K_2 = K_2K_1, && K_iK_i^{-1} = 1 = K_i^{-1}K_i \quad(i=1,2)\\
  & K_1 E K_1^{-1} = v E, && K_2 E K_2^{-1} = v^{-1} E \\
  & K_1 F K_1^{-1} = v^{-1} F, && K_2 F K_2^{-1} = v F \\
  & EF - FE = \frac{K-K^{-1}}{v-v^{-1}},\quad && \text{where $K:= K_1K_2^{-1}$}.
  \label{e:UU-gl-d}
\end{alignat}
\end{subequations}
The algebra $\UU(\gl_2)$ is a Hopf algebra with counit $\epsilon:
\UU(\gl_2) \to \Bbbk$ and coproduct $\Delta: \UU(\gl_2) \to \UU(\gl_2)
\otimes \UU(\gl_2)$ given on generators by
\begin{subequations}\label{e:coprod}
\begin{gather}
  \Delta(E) = E \otimes 1 + K \otimes E, \quad
  \Delta(F) = F \otimes K^{-1} + 1 \otimes F\\
  \Delta(K_i) = K_i \otimes K_i \quad (i=1,2)\\
    \epsilon(E) = \epsilon(F) = 0, \quad \epsilon(K_i) = 1 \quad(i=1,2).
\end{gather}
\end{subequations}
We omit the definition of the antipode as it isn't needed here.


The algebra $\UU(\fsl_2)$ is the subalgebra of $\UU(\gl_2)$ generated
by $E$, $F$, and $K^{\pm1}$. 
Its generators satisfy the defining relations
%
%
$K K^{-1} = 1 = K^{-1} K$, $K E K^{-1} = v^2 E$, $K F K^{-1} = v^{-2}
F$ along with relation~\eqref{e:UU-gl-d}. By restriction,
$\UU(\fsl_2)$ inherits a Hopf algebra structure from that on
$\UU(\gl_2)$.

Our conventions are the same as in \cite{Lusztig:89},
which slightly modified the original definitions of
e.g.~\cites{Drinfeld,Jimbo:85}.  By letting $v$ be an element of the
ground field instead of taking it to be an indeterminate, we are
following the approach given in the first few chapters of Jantzen's
book \cite{Jantzen}; the books \cites{Lusztig,Kassel} are also useful
general references.

From now on, $\UU = \UU(\gl_2)$ or $\UU(\fsl_2)$. By a $\UU$-module we
always mean a type $\mathbf{1}$ $\UU$-module; see
\cite{Jantzen}*{5.3}.  A vector $m$ in a $\UU$-module $M$ is said to
be a weight vector of weight $w \in \Z$ if $K m = v^w m$. Let $V(n)$
be the simple $\UU$-module of dimension $n+1$, with standard basis
$x_i$ ($i=0,\dots, n$) of weight vectors such that
\begin{equation}
  K x_i = v^{n-2i} x_i,\quad
  F x_i = x_{i+1} \text{ if $i<n$},\quad \text{and} \quad
  Fx_n = 0.
\end{equation}
See \cite{Jantzen}*{2.6} or \cite{Kassel}*{Thm.~VI.3.5} for further
details. In particular, $V := V(1)$ is the natural (or ``vector'')
module. We make $\UU$ act on $V^{\otimes n}$ by means of the iterated
coproduct $\Delta^{(n)}: \UU \to \UU^{\otimes n}$, defined inductively
by
\[
\Delta^{(2)} = \Delta, \qquad \Delta^{(k+1)} = (\Delta \otimes
1^{\otimes (k-1)}) \Delta^{(k)} \text{ if } k \ge 2.
\]
Thus, $V^{\otimes n}$ is a $\UU$-module. Since $v$ is not a root of
unity, it is a semisimple $\UU$-module. (See e.g.~\cite{Jantzen}*{5.17
  and 6.26} or \cite{Kassel}*{Thm.~VII.2.2}.)

In order to define an action of $\TL_n(\delta)$ on $V^{\otimes n}$,
consider the linear map $\xi: V \otimes V \to V\otimes V$ defined on
basis elements by
\[
x_{0,0} \mapsto 0, \quad x_{0,1} \mapsto v^{-1} x_{0,1} - x_{1,0}, \quad
x_{1,0} \mapsto - x_{0,1} + v x_{1,0}, \quad x_{1,1} \mapsto 0
\]
where we write $x_{i,j} := x_i \otimes x_j$ to simplify notation. Then
$\xi^2 = (v+v^{-1}) \xi$, so $(\pm \xi)^2 = \pm(v+v^{-1}) (\pm \xi)$.
If $\delta = \pm(v+v^{-1})$, let $e_i$ in $\TL_n = \TL_n(\delta)$ act
on $V^{\otimes n}$ as the linear map
\begin{equation}\label{e:e_i}
  e_i = 1^{\otimes i-1} \otimes (\pm \xi) \otimes 1^{\otimes
    n-i-1}
\end{equation}
with $\pm \xi$ operating on the copy of $V \otimes V$ embedded in
tensor positions $i,i+1$. It turns out that $\xi$ is a $\UU$-module
endomorphism of $V \otimes V$, so it follows that $e_i$ is a
$\UU$-module endomorphism of $V^{\otimes n}$.

\begin{lem}\label{l:TL-action}
The $e_i$ ($i=1,\dots, n-1$ defined by \eqref{e:e_i} satisfy the
defining relations of $\TL_n(\delta)$, with $\delta= \pm(v+v^{-1})$.
So we have defined an action of $\TL_n(\delta)$ on $V^{\otimes n}$,
and this action commutes with the action of $\UU$. 
\end{lem}

\begin{proof}
The proof is by a tedious yet elementary calculation, which we omit.
(It suffices to check one of the sign choices for $\delta$, thanks to
the isomorphism $\TL_n(\delta) \cong \TL_n(-\delta)$.)  That the
action commutes with the action of $\UU$ follows from the fact that
the $e_i$ are $\UU$-module morphisms.
\end{proof}

\begin{rmk}\label{r:orthog}
%
(i) The operator $\xi$ may be derived from the $R$-matrix formalism,
  following Jimbo \cite{Jimbo}.

(ii) The standard inner product on $V$ (defined by $\bil{x_i}{x_j} =
\delta_{i,j}$) extends to $V \otimes V$ by setting
$\bil{x_{i,j}}{x_{k,l}} = \delta_{i,k}\delta_{j,l}$.  The orthogonal
projection of $V\otimes V$ onto the line spanned by $z_0:= x_{0,1} - v
x_{1,0}$ is given by the matrix
\[
P = \frac{1}{v+v^{-1}} \,
\begin{bmatrix}
  0&0&0&0\\
  0 & v^{-1} & -1 & 0\\
  0 & -1 & v & 0\\
  0&0&0&0
\end{bmatrix}
\]
with respect to the ordered basis $x_{0,0}$, $x_{0,1}$, $x_{1,0}$,
$x_{1,1}$.  (Note that $v+v^{-1} \ne 0$, since $v$ is not a root of
unity.) So $\xi = (v+v^{-1}) P$, where we identify matrices with the
linear maps they define, as usual. The image of $\xi$ is $\Bbbk z_0
\cong V(0)$ and its kernel is the $\Bbbk$-span of $z_1:= x_{0,0}$,
$z_2:= x_{0,1} + v x_{1,0}$, and $z_3:= x_{1,1}$. This span is
isomorphic to $V(2)$, so we have $V \otimes V \cong V(0) \oplus V(2)$.
The direct summands are orthogonal.
\end{rmk}

If $v$ is transcendental, the next result may be deduced from 
\cite{Jimbo} with some work, using the results of Section
\ref{s:Hecke} and the notation of Remark~\ref{r:HH2}. In Jimbo's
result, the algebra $\HH_n(v^{-1},-v)$ acts (usually non-faithfully)
on $V^{\otimes n}$, and by passing to the corresponding quotient by
the kernel of that action one obtains a faithful action of
$\TL_n(\delta)$.  As \cite{Jimbo} did not include a proof, we will
sketch a proof in a slightly more general context.  First, we recall
the notation
\[
\Lambda(n) = \{n, n-2, \dots, n-2l\}, \quad\text{where } l = \lfloor n/2
\rfloor
\]
from the preceding section.

\begin{thm}[Schur--Weyl duality if $v$ is not a root of unity]\label{t:SWD}
Let $\Bbbk$ be a field, and $0 \ne v \in \Bbbk$. Assume that $v$ is
not a root of unity. Let $\delta = \pm(v+v^{-1})$. Then the above
action of $\TL_n(\delta)$ on $V^{\otimes n}$ commutes with the action
of $\UU = \UU(\gl_2)$ or $\UU(\fsl_2)$, and these commuting actions
induce algebra surjections
\[
\UU \to \End_{\TL_n(\delta)}(V^{\otimes n}), \quad
\TL_n(\delta) \to \End_{\UU}(V^{\otimes n})
\]
the second of which is actually an isomorphism.  Furthermore,
\[
V^{\otimes n} \cong \textstyle \bigoplus_{k \in \Lambda(n)} V(k) \otimes H(k)
\]
is a decomposition into simple $\UU \otimes \TL_n(\delta)$-modules.
\end{thm}

\begin{proof}
By Lemma~\ref{l:TL-action}, the actions of $\UU$ and $\TL_n(\delta)$
commute. The action of $\TL_n(\delta)$ induces a representation
(algebra morphism)
\begin{equation}\label{e:induced-map}
  \TL_n(\delta) \to \End_\UU(V^{\otimes n}).
\end{equation}
We claim that the dimensions of $\End_\UU(V^{\otimes n})$ and
$\TL_n(\delta)$ are equal, and thus \eqref{e:induced-map} is an
isomorphism.  Let $m_{n,\lambda}:= [V^{\otimes n}: V(\lambda)]$ be the
composition factor multiplicity of $V(\lambda)$ in $V^{\otimes n}$, as
$\UU$-modules. Then it suffices to show that
\[
m_{n,\lambda} = \dim_\Bbbk H(\lambda)
\quad \text{for each $\lambda \in \Lambda(n)$}.
\]
Applying the quantum Clebsch--Gordon rule (see \cite{Jimbo:85} or
\cite{Kassel}*{\S VII.7}), for any $\lambda$ in $\Lambda(k-1)$ we have
\[
V(\lambda)\otimes V = V(\lambda)\otimes V(1) \cong
\begin{cases}
  V(\lambda+1) \oplus V(\lambda-1) & \text{ if } \lambda > 0\\
  V(\lambda+1) = V(1) & \text{ if } \lambda=0.
\end{cases}
\]
It follows that the multiplicity $m_{n,\lambda}$ is equal to the
number of paths from $\emptyset$ to $(n-p,p)$ in the Bratteli diagram
(see Figure~\ref{f:Bratteli}), where $\lambda = n-2p$.  By
Lemma~\ref{l:bijections}, $m_{n,\lambda}$ is the number of
half-diagrams on $n$ nodes. This implies the claim.

The final step is to apply Jacobson's density theorem
\cites{Jacobson,Lang} to conclude that $V^{\otimes n}$
satisfies the double-centralizer property as a $\UU$-module. In other
words, the natural map
\[
\UU \to \End_{\TL_n(\delta)}(V^{\otimes n}) \quad \text{is surjective}.
\]
Here we identify $\TL_n(\delta)$ with $\End_\UU(V^{\otimes n})$ by
means of the isomorphism in \eqref{e:induced-map}.  This completes the
proof of both surjectivity statements in Theorem~\ref{t:SWD}. The
final claim in Theorem~\ref{t:SWD} is a standard consequence of the
semisimplicity of $V^{\otimes n}$, so the proof is complete.
\end{proof}

\begin{rmk}
If $[n]_v^! \ne 0$ and $\delta = \pm(v+v^{-1})\ne 0$, the paper
\cite{DG:orthog} replaces $\UU$ by the quantum Schur algebra
$S_v(2,n)$ of \cite{DJ3} in homogeneous degree $n$ (which acts
faithfully) and constructs an orthogonal basis of maximal vectors
(with respect to the induced inner product on tensor space). This
gives an orthogonal decomposition of tensor space $V^{\otimes n}$ as
$S_v(2,n)$-modules.  This leads to a combinatorial proof of a
slightly stronger version of Theorem~\ref{t:SWD} in which the
hypothesis that $v$ is not a root of unity is weakened to the
condition $[n]_v^! \ne 0$.
\end{rmk}

Recall that a nonzero vector in a $\UU$-module is \emph{maximal} if it
is killed by the generator $E$.  We have the following consequence of
Theorem~\ref{t:SWD}, which realizes the simple $\TL_n(\delta)$-modules
in tensor space.

\begin{cor}
  Suppose that $0 \ne v \in \Bbbk$ is not a root of unity.  Let
  $\mathrm{Max}(\lambda)$ be the space of all maximal vectors of
  weight $\lambda$ in $V^{\otimes n}$, for $\lambda \in
  \Lambda(n)$. Then $\mathrm{Max}(\lambda) \cong H(\lambda)$, as
  $\TL_n(\delta)$-modules, where $\delta = \pm(v+v^{-1})$.
\end{cor}

\begin{proof}
The action of any $e_i$ preserves the weight of a weight vector, so
$\mathrm{Max}(\lambda)$ is a $\TL_n(\delta)$-submodule.  Each nonzero
vector in $\mathrm{Max}(\lambda)$ generates a copy of $V(\lambda)$ in
$V^{\otimes n}$, as $\UU$-modules, so $\Hom_\UU(V^{\otimes n},
V(\lambda)) \cong \mathrm{Max}(\lambda)$, as $\TL_n(\delta)$-modules.
But also $\Hom_\UU(V^{\otimes n}, V(\lambda)) \cong H(\lambda)$
follows from the last statement in Theorem~\ref{t:SWD}.
\end{proof}

\begin{rmk}
Takeuchi \cite{Takeuchi} introduced a two-parameter version
$\UU_{q_1,q_2}(\gl_m)$ of the quantized enveloping algebra of
$\gl_m$. Under suitable hypotheses, Benkart and Witherspoon \cite{BW}
extended Jimbo's Schur--Weyl duality to commuting actions of
$\UU_{q_1,q_2}(\gl_m)$ and $\HH_n(q_1,q_2)$ on tensor space, where here
$\HH_n(q_1,q_2)$ is the two-parameter version on the Iwahori--Hecke
algebra defined in Remark~\ref{r:HH2}.
\end{rmk}

By excluding roots of unity, Theorem~\ref{t:SWD} does not apply to the
important case $\delta=0$, which corresponds to $v = \pm \sqrt{-1}$ in
$\Bbbk$.  However, Jimbo's Schur--Weyl duality was generalized to
arbitrary $v \ne 0$ in \cite{DPS:98}; see also \cite{Martin:92}. Thus,
the first statement in Theorem~\ref{t:SWD} holds more generally.  For
the sake of completeness, we explain how to derive this more general
statement from \cites{DPS:98,Haerterich}. For this it is necessary to
work with Lusztig's divided power form of the quantized enveloping
algebras, which we now define, following \cite{Jantzen}*{Chap.~11}.
Let $t$ be an indeterminate, $\ZZ = \Z[t,t^{-1}]$ the ring of integral
Laurent polynomials. The field of fractions of $\ZZ$ is $\Q(t)$.  Let
$\UU_{\Q(t)}(\gl_2)$ be the algebra over $\Q(t)$ defined by the
generators and relations \eqref{e:UU-gl} but with $v$ replaced by $t$.
Then $\UU_{\Q(t)}(\fsl_2)$ is the $\Q(t)$-subalgebra generated by $E$,
$F$ and $K$.  Let $\UU_\ZZ(\gl_2)$ (resp., $\UU_\ZZ(\fsl_2)$) be the
$\ZZ$-subalgebra of $\UU_{\Q(t)}(\gl_2)$ generated by the quantum
divided powers
\[
E^{(j)} := E^j/[j]^!_t, \quad F^{(j)} := F^j/[j]^!_t \qquad \text{(for
  $j \ge 0$)}
\]
and the $K_i^{\pm 1}$ for $i=1,2$ (resp., $K^{\pm1}$).  For any
commutative ring $\Bbbk$, fix an invertible element $v$ in
$\Bbbk$ and make $\Bbbk$ into an $\ZZ$-algebra via the ring morphism
sending $t^{\pm1} \mapsto v^{\pm1}$. Following Lusztig,
we define
\[
\UU(\gl_2) = \UU_\ZZ(\gl_2) \otimes_\ZZ \Bbbk \quad\text{and}\quad
\UU(\fsl_2) = \UU_\ZZ(\fsl_2) \otimes_\ZZ \Bbbk .
\]
These algebras are the divided power forms mentioned above. By a
standard abuse of notation, we identify any generator $g$ of $\UU_\ZZ$
with its image $g \otimes 1$ in $\UU$, for either version of $\UU$.
The Lusztig form makes sense at roots of unity, for instance if $v=1$.
The algebras $\UU_\ZZ(\gl_2)$, $\UU_\ZZ(\fsl_2)$ are Hopf algebras
under the natural restriction of the Hopf algebra maps, so
$\UU(\gl_2)$ and $\UU(\fsl_2)$ are also Hopf algebras.

If $v$ is not a root of unity then it
is well known that the versions of $\UU = \UU(\gl_2)$ or $\UU(\fsl_2)$
defined at the beginning of this section are isomorphic to their
Lusztig divided power forms, so the notation is unambiguous.

Now consider $V_{\Q(t)}(m)$, the simple highest weight module for
$\UU_{\Q(t)}$ of highest weight $m$.  Let $V_\ZZ(m)$ be the $\ZZ$-span
of the standard basis $\{x_i\}_{i=0}^m$ of $V_{\Q(t)}(m)$.  One easily
checks that this is an admissible lattice in $V_{\Q(t)}(m)$, that is, a
$\UU_\ZZ$-submodule which is free over $\ZZ$ and which is the direct
sum of its weight spaces, such that $V_{\Q(t)}(m) \cong V_{\ZZ}(m)
\otimes_\ZZ \Q(t)$. Then
\[
V(m) := V_{\ZZ}(m) \otimes_\ZZ \Bbbk
\]
is a $\UU$-module, where $\UU = \UU(\gl_2)$ or $\UU(\fsl_2)$. It is the
quantized Weyl module of highest weight $m$. In general it is no longer a
simple module.

From now on $\Bbbk$ is a field. For $\delta = \pm(v+v^{-1})$, the
action of $\TL_n(\delta)$ defined by \eqref{e:e_i} still makes sense
even if $v$ is a root of unity. If $v+v^{-1} \ne 0$ then $\xi$ is an
orthogonal projection, and $V \otimes V \cong V(0) \oplus V(2)$, and
both Weyl modules $V(0)$, $V(2)$ are simple, as in Remark
\ref{r:orthog}(ii).  However, if $v+v^{-1}=0$ then $V(2)$ is no longer
simple, as it has a simple submodule $\Bbbk z_2 \cong V(0)$ (notation
of Remark \ref{r:orthog}(ii)) and $V \otimes V$ is an indecomposable
tilting module \cite{Andersen} of highest weight $2$. In this case the
map $\xi$ in \eqref{e:e_i} satisfies $\xi^2=0$, the kernel of $\xi$ is
isomorphic to $V(2)$, and its image is isomorphic to $V(0)$.

\begin{thm}[Schur--Weyl duality, general case]
  Let $0 \ne v \in \Bbbk$ where $\Bbbk$ is a field.  Set $\delta =
  \pm(v+v^{-1})$. Let $V = V(1)$. Then for $\UU = \UU(\gl_2)$ or
  $\UU(\fsl_2)$, the commuting actions of $\UU$ and $\TL_n(\delta)$ on
  $V^{\otimes n}$ satisfy Schur--Weyl duality; that is, each of the
  induced algebra maps
  \[
  \UU \to \End_{\TL_n(\delta)}(V^{\otimes n}), \quad
  \TL_n(\delta) \to \End_{\UU}(V^{\otimes n})
  \]
  is surjective. Moreover, the latter map is an isomorphism.  
\end{thm}

\begin{proof}[Proof sketch]
By a special case of \cite{DPS:98}, there are commuting actions of
$\UU(\gl_2)$ and $\HH_n(-1,v^2)$ on $V^{\otimes n}$ which satisfy
Schur--Weyl duality, where we employ the two-parameter notation as in
Remark~\ref{r:HH2}. So the induced algebra maps
\[
\UU(\gl_2) \to \End_{\HH_n(-1,\,v^2)}(V^{\otimes n}), \quad
\HH_n(-1,v^2) \to \End_{\UU(\gl_2)}(V^{\otimes n})
\]
are both surjective. We need to argue that $\UU(\gl_2)$ may be
replaced by $\UU(\fsl_2)$ and $\HH_n(-1,v^2)$ by $\TL_n(\delta)$
without altering the statement. The first such replacement is trivial,
and follows from the fact that the images of $\UU(\fsl_2)$ and
$\UU(\gl_2)$ in $\End_\Bbbk(V^{\otimes n})$ are the same, because the
two algebras differ by a generator that acts as scalars on $V^{\otimes
  n}$.

The other replacement is not so trivial. First, argue that
$\HH_n(-1,v^2)$ may be replaced by $\HH_n(-v^{-1},v)$; see Remark
\ref{r:HH2}. Then apply \cite{Haerterich} to see that the dimensions
of the kernel and image of the Hecke algebra action is invariant
regardless of the choice of field or specialization $t \mapsto v$.
It follows that the image of that action is isomorphic to
$\TL_n(\delta)$. The result now follows by Corollary~\ref{c:TL-quo}.
\end{proof}

\begin{rmk}
If we let $\Bbbk=\C$ and take $v = \pm \sqrt{-1}$ then we obtain a
version of Schur--Weyl duality that applies to $\TL_n(0)$, showing in
particular that $\TL_n(0) \cong \End_{\UU}(V^{\otimes n})$. The same
remark applies more generally in any field containing a square root of
$-1$.
\end{rmk}

\subsection*{Acknowledgments}
The authors are grateful to Fred Goodman and the anonymous referees
for helpful comments and suggestions.


\end{document}